\documentclass[reqno,12pt,20pt]{amsart}
\usepackage[ansinew]{inputenc}
\usepackage[english]{babel}
\usepackage[all]{xy}
\usepackage{amsmath,latexsym,amssymb,verbatim}
\input arrow.tex

\newcommand{\La}{{\operatorname{L}}}
\newcommand{\tr}{{\operatorname{tr}}}
\newcommand{\Tr}{{\operatorname{Tr}}}
\newcommand{\lc}{{\operatorname{LC}}}

\newcommand{\LL}{{\mathbb L}}
\newcommand{\M}{{\mathcal M}}
\newcommand{\E}{{\mathcal E}}

\newcommand{\Lc}{{\mathcal L}}

\newcommand{\ZZ}{{\mathbb Z}}
\newcommand{\TT}{{\mathbb T}}
\newcommand{\QQ}{{\mathbb Q}}

\newcommand{\RR}{\mathbb R}

\DeclareMathOperator{\limi}{{lim}}
\newcommand{\plim}[1]{\,\underset{#1}{\underset{\leftarrow}{\limi}}\,}

\newcommand{\wt}{\widetilde}

\newcommand{\GShv}{\operatorname{{\it G}-Shv}}
\newcommand{\LLc}{\operatorname{{\mathbb {LC}}}}
\newcommand{\Cyl}{\operatorname{Cyl}}

\newcommand{\enumera}{\begin{enumerate}}
\newcommand{\eenumera}{\end{enumerate}}
\newcommand{\C}{{\mathcal C}}
\newcommand{\Q}{{\mathcal Q}}

\DeclareMathOperator{\Hom}{{Hom}}

\DeclareMathOperator{\Id}{{Id}}
\newcommand{\ilim}[1]{\,\underset{#1}{\underset{\to}{\limi}}\,}
\newcommand{\alineas}[1]{\begin{array}{#1}}
\newcommand{\alinea}{\begin{array}{l}}
\newcommand{\ealinea}{\end{array}}
\newcommand{\ealineas}{\end{array}}

\newcommand{\co}{\underline{\operatorname{cos}}}
\newcommand{\pun}{{\scriptscriptstyle \bullet}}

\DeclareMathOperator{\HHom}{\underline{Hom}}

\newcommand{\Qcoh}{\operatorname{Qcoh}}
\newcommand{\shf}{\underline{\operatorname{shf}}}
\newcommand{\Coqc}{\operatorname{Coqc}}
\newcommand{\Abe}{\operatorname{Ab}}
\newcommand{\Mod}{\operatorname{Mod}}
\newcommand{\Shv}{\operatorname{Shv}}
\newcommand{\Coshv}{\operatorname{Coshv}}

\newcommand{\LConst}{\operatorname{LC}}

\DeclareMathOperator{\Aut}{{Aut}}
\DeclareMathOperator{\Tor}{{Tor}}
\DeclareMathOperator{\Ext}{{Ext}}
\DeclareMathOperator{\Qc}{{Qc}}

\DeclareMathOperator{\id}{{id}}

\newtheorem{lemma}{Lemma}[section]
\newtheorem{thm}[lemma]{Theorem}
\newtheorem{prop}[lemma]{Proposition}
\newtheorem{coro}[lemma]{Corollary}
\theoremstyle{definition}
\newtheorem{defn}[lemma]{Definition}

\newtheorem{rem}[lemma]{Remark}

\newtheorem{ejems}[lemma]{Examples}

\begin{document}

\title{Homology of sheaves  via Brown representability}
%\author{Amelia \'{A}lvarez, Fernando Sancho, and Pedro Sancho}
%\date{10-1-2005}

\author{Fernando Sancho de Salas}

\address{ Departamento de
Matem\'aticas and Instituto Universitario de F\'isica Fundamental y Matem\'aticas (IUFFyM)\newline
Universidad de Salamanca\newline  Plaza de la Merced 1-4\\
37008 Salamanca\newline  Spain}
\email{fsancho@usal.es}



\thanks {The  author was supported by research project MTM2017-86042-P (MEC)}
 
\begin{abstract} We give an elementary construction of homology of sheaves from Brown representability for the dual and see how its main properties are derived easily from the construction. Comparison with Poincar\'e-Verdier duality and with homology of groups are also developed.
\end{abstract}

%\date{January 1, 1994 and, in revised form, June 22, 1994.}

%\dedicatory{This paper is dedicated to our authors.}

\maketitle

\section*{Introduction}

On an arbitrary topological space, cohomology is defined for any sheaf and studied within the framework of the theory of derived functors, whereas homology is usually defined only for constant or locally constant coefficients and thus does not fit within such a framework. However, on finite topological spaces one may in fact define the homology and cohomology of any sheaf and both of these constructions are developed within the framework of the theory of derived functors. In other words, on a finite topological space it is possible to consider the homology groups $H_i(X, F)$ with coefficients on any sheaf $F$, and in this case $H_i(X,\quad)$ is the $i$-th (left-)derived functor of $H_0(X,\quad)$. This was already pointed out by Deheuvels (\cite{Deheuvels}) and a systematic treatment may be found in \cite{Sanchoetal} (see also \cite{Curry}). Of course, on a finite topological space (more generally, on an Alexandroff space), the category $\Shv(X)$ of sheaves of abelian groups on $X$ has enough projectives and this fails to be true for general (and most common) topological spaces.

Another context where one has an homology theory for sheaves is on locally compact and Haussdorff spaces, where one has Borel-Moore homology, developed in \cite{Bredon}; here we mean Borel-Moore homology with compact supports, which is the one with the  expected properties of an homology theory (as singular homology). As it is mentioned in op. cit., Bore-Moore homology is, more than a homology theory, a co-cohomology theory, and its construction needs a detour through the theory of cosheaves.

Another direction in homology theory is to construct homology of cosheaves, instead of sheaves. The reason of this is that one expects homology to be covariant with respect to open inclusions, i.e., an open inclusion $V\subset U$ must induce a morphism $H_i(V,F)\to H_i(U,F)$, hence cosheaves seem better adapted for this purpose. The problem here is that the category of cosheaves is not as good as that of sheaves, a main problem being the existence of a cosheafification procedure of a precosheaf (the analog of the sheafification of a presheaf) (see \cite{Prasolov}, \cite{Prasolov2}).

The first aim of this paper is  to give an elementary construction of an homology theory of sheaves, under mild local hypothesis on the topological space, by using Brown representability for the dual, and which agrees with that of finite spaces and that of locally compact Haussdorff ones. The  second aim is to show that all the properties that one would desire follow from the construction in a transparent way. Let us give some details. Let $X$ be any topological space and $\pi\colon X\to \{*\}$ the projection to a point; let us denote by $D(X)$ the (unbounded) derived category of the category of sheaves on $X$ and $D(\ZZ)$ the derived category of abelian groups. Let us consider the inverse image functor \[ \pi^{-1}\colon D(\ZZ)\to D(X).\]
This functor has a right adjoint, the derived global sections functor $\RR\Gamma(X,\quad)$. This is cohomology. We shall use Brown representability for the dual to prove  (Theorem \ref{derivedproduct}) that $\pi^{-1}$ has also a left adjoint, denoted by $\LL(X,\quad)$, under mild local hypothesis on $X$: we shall assume that $X$ is locally connected and locally cohomologically trivial (see Definition \ref{clc}); on locally compact Haussdorff spaces the latter condition can be weakened to a locally cohomologically connected condition (see Remark \ref{loc-compact}). We shall define the homology groups with coefficients on a sheaf $F$ as $H_i(X,F):=H_i\LL(X,F)$. It is important to remark that one has to use the unbounded derived category (instead of $D^-(X)$, as one would be tempted) in order to have Brown representability available, in the same way that Neeman simplified Grothendieck duality via Brown representability (\cite{Neeman}) by considering the unbounded derived category of quasi-coherent sheaves.

Once the homology groups $H_i(X,F)$ are defined, we show  in an easy way its main properties: (a) Long exact sequence associated to an exact sequence of sheaves, (b) Duality between homology and cohomology (Theorem \ref{homology-cohomology}), 
(c) Covariance with respect to con\-ti\-nuous maps $X\to Y$ (Proposition \ref{homology-cohomology}),  (d) Mayer Vietoris sequences (Propositions \ref{MV1} and \ref{MV2}), (e) Vietoris theorem and homotopy invariance (Theorem \ref{h-inv} and Corollary \ref{h-inv2}), (f) Excision (Proposition \ref{excision}),(g) Universal coefficients formula (Theorem \ref{universalcoefficients}), (h)   Cap product (Theorem \ref{cap}), (i) K\"{u}nneth formula (Theorem \ref{Kunneth}).

In section \ref{Sec-Poincare-Verdier} we make a comparison with Poincar\'e-Verdier duality in a very general form. Let $X$ be a topological space with a duality theory: an exact functor $\omega \colon D(X)\to D(\ZZ)$ that admits a right adjoint $\omega^!$ (for example, $X$ a locally compact Haussdorff space and $\omega=\RR\Gamma_c(X,\quad)$ the cohomology with compact support);  let us denote $D_X^\omega:=\omega^!(\ZZ)$ the ``dualizing complex'' and $H^i_\omega(X,F)=H^i(\omega(F))$. Then, there is a natural morphism (Proposition \ref{PV-comparison})
\[ H^{-i}(X,F\overset\LL\otimes D_X^\omega)\to H_i(X,F)\] and we determine when this morphism is an isomorphism for any $F$ (Theorem \ref{Poincare-Verdier}). We also characterize homological manifolds   in section \ref{homologicalmanifolds}, as Bredon does on locally compact Haussdoff spaces (\cite{Bredon})

Section \ref{homologyofgroups} is devoted to the comparison between homology of sheaves a homology of groups, as it happens for cohomology of sheaves and cohomology of groups. This requires the introduction of a functor from sheaves to locally constant sheaves which is the homological analog of the quasicoherator functor on schemes. This functor is again constructed from Brown representability.

Finally, section \ref{cosections} is devoted to the study of the underived version of homology, i.e, of the functor of cosections $\La(X,F)=H_0(X,F)$. It is nothing but the left adjoint of the inverse image $\pi^{-1}\colon \Mod(\ZZ)\to\Shv(X)$, where $\Mod(\ZZ)$ denotes the category of abelian groups. We shall see (Theorem \ref{enough-L-acyclics}) that there are enough $\La(X,\quad)$-acyclics and then $\LL(X,\quad)$ may be viewed as a left derived functor of $\La(X,\quad)$. We also see in section \ref{sheaves-cosheaves} how the functor of cosections relates  sheaves and cosheaves in  a natural way.

\section{Cosections}\label{cosections}

\subsection{Notations and conventions}
 
Let $X$ be a topological space. By a sheaf on $X$ we always mean a sheaf of abelian groups. We shall denote by $\Shv(X)$ the category of sheaves   on $X$. If $f\colon X'\to X$ is a continuous map and $F$ is a sheaf on $X$, $f^{-1}F$ denotes the inverse image sheaf by $f$. If $j\colon U\hookrightarrow X$ is an open subset, we also denote $j^{-1}F=F_{\vert U}$ and we denote by $j_!\colon \Shv(U)\to\Shv(X)$ the extension by zero functor, which is left adjoint of $j^{-1}$. We shall denote $F_U=j_! F_{\vert U}$.

Let $X$ be a topological space, $\pi_X\colon X\to \{*\}$ the projection to a point,   and $\Mod(\ZZ)$ the category of abelian groups.

\begin{prop}\label{directprodconstant} If $X$ is locally connected, the functor 
\[ \pi^{-1}_X \colon \Mod(\ZZ)\to \Shv(X)\] commutes with direct products.
\end{prop}

\begin{proof} Let $\{G_i\}_{i\in I}$ be a family of abelian groups. The natural morphism $\pi_X^{-1}(\prod_iG_i)\to\prod_i(\pi_X^{-1}G_i)$ is an isomorphism after taking sections on a connected open subset. Since $X$ is locally connected, it is a stalkwise isomorphism.
\end{proof}
 
\begin{thm} Let $X$ be a locally connected space. The functor $\pi^{-1}_X \colon \Mod(\ZZ)\to \Shv(X)$ has a left adjoint $$\La(X,\quad) \colon \Shv(X) \to \Mod(\ZZ) .$$ Thus, for any sheaf $F$ on  $X$ and any abelian group $G$ one has
\[ \Hom_\ZZ(\La(X,F),G)=\Hom_{\Shv(X)}(F,\pi_X^{-1}G).\]
\end{thm}

\begin{proof} Since $\pi_X^{-1}$ is exact and commutes with direct products, it has a left adjoint.
%Let $F$ be a sheaf on $X$. The functor 
%\[ \aligned \Mod(\ZZ)&\to\Mod(\ZZ)\\ G&\mapsto \Hom(F,\pi_X^{-1}G)\endaligned \] is left exact and takes direct products  to direct products by Proposition \ref{directprodconstant}. Hence it is representable: there exists an abelian group $(\pi_X)_!F$ and a functorial isomorphism 
%\[ \Hom((\pi_X)_!F,G) = \Hom(G,\pi_X^{-1}G).\] A morphism of sheaves $F\to F'$ induces a morphism of functors $\Hom(\quad,\pi_X^{-1}F)\to \Hom(\quad,\pi_X^{-1}F)$, hence (by Yoneda) a morphism between the representants $(\pi_X)_!F\to (\pi_X)_!F'$. Thus $(\pi_X)_!$ is a functor which is left adjoint of $\pi_X^{-1}$.
\end{proof}
 {\bf\ \ Notation}. As it is usual, the constant sheaf $\pi_X^{-1}G$ shall be denoted simply by $G$.
 
\begin{defn}  The group $\La(X,F) $ shall be called {\em cosections} of $F$ on $X$. 
\end{defn}
 
 For the rest   section 1, the space is always assumed to be locally connected.
 
\begin{prop}[Explicit computation] Let $I$ be the set of connected and non-empty open subsets of $X$. For each $i\in I$, let $U_i$ be the connected open subset corresponding to $i$. We say that $i\leq j$ if $U_i\supseteq U_j$. Then
\[ \La(X,F)= \ilim{i\in I} F(U_i).\] For example, for any constant sheaf $G$:
\[ \La(X,G)=G^{\oplus \pi_0(X)}.\]
\end{prop}

\begin{proof} Let $F$ be a sheaf. Since $X$ is locally connected, for any open subset $U$, $F(U)=\prod_{j\in\pi_0(U)} F(U_j)$. In other words, $F$ is uniquely determined by its value on connected open subsets. Then, if $G$ is a constant sheaf, a morphism $F\to G$ is equivalent to giving a collection of morphisms $\phi_i\colon F(U_i)\to G(U_i)$ which are compatible with restrictions. Since $G(U_i)=G$ and the restrictions morphisms $G(U_i)\to G(U_j)$ are the identity, we conclude that a morphism $F\to G$ is a collection of morphisms $\phi_i\colon F(U_i)\to G$ such that
\[\xymatrix{ F(U_i)\ar[rr]^{\phi_i}\ar[rd] & & G\\ & F(U_j)\ar[ur]_{\phi_j}}\] is commutative for any $i\leq j$. This is a morphism $\ilim{i\in I} F(U_i)\to G$.
\end{proof}
 
\begin{prop}From representability, it follows \begin{enumerate} \item $\La(X,\quad)$ is right exact: an epimorphism $F\to F'\to 0$ induces an epimorphism $\La(X,F)\to \La(X,F')\to 0.$
\item $\La(X,\quad)$ commutes with direct sums: $ \La(X,\oplus_i F_i)=\oplus_i 
\La(X,F_i)$. More generally, $\La(X,\quad)$ commutes with direct limits.
\end{enumerate}
\end{prop}

\begin{prop} Let $f\colon X\to Y$ be a continuous map between locally connected spaces. For any sheaf $F$ on $Y$ one has a natural morphism
\[ f_*\colon \La (X,f^{-1}F)\to \La (Y,F).\] If $g\colon Y\to Z$ is another continuous map, then $(g\circ f)_*= g_*\circ f_*$. Moreover $(\id_X)_*=\id$.
\end{prop}

\begin{proof} For any abelian group $G$ one has a natural morphism $$\Hom_{\Shv(Y)}(F, G)\to \Hom_{\Shv(X)}(f^{-1}F,f^{-1} G)= \Hom_{\Shv(X)}(f^{-1}F, G)$$ i.e., a morphism
\[ \Hom_\ZZ(\La(Y,F),G)\to \Hom_\ZZ(\La(X,f^{-1}F),G)\] and then the desired morphism $f_*\colon \La (X,f^{-1}F)\to \La (Y,F)$. The equality $(g\circ f)_*= g_*\circ f_*$ follows from the construction and from the equality $f^{-1}\circ g^{-1}=(g\circ f)^{-1}$;  the equality $(\id_X)_*=\id$ from the equality $(\id_X)^{-1}=\id$.
\end{proof}

\begin{prop}\label{F_U} For any open subset $j\colon U\hookrightarrow X$ and any  sheaf $F'$ on $U$, one has
\[ \La(U,F')=\La(X,j_!F').\] Hence, for any sheaf $F$ on $X$ one has
\[ \La(U,F_{\vert U})=\La(X,F_U)\] and the morphism $j_*\colon \La(U,F_{\vert U})\to \La(X,F )$ is obtained by applying $\La(X,\quad)$ to the morphism $F_U\to F$. 
\end{prop}

\begin{proof} The functor $j^{-1}\colon \Shv(X)\to\Shv(U)$ is right  adjoint of $j_!\colon \Shv(U)\to\Shv(X)$. The equality $\pi_U^{-1}=j^{-1}\circ \pi_X^{-1}$ induces, by adjunction, an equality $\La(U,\quad)=\La(X,\quad)\circ j_!$.    The statement about $j_*$ follows from its construction. 
\end{proof}

\subsection{Sheaves and cosheaves}\label{sheaves-cosheaves}

\begin{defn} (\cite[Ch.V, Def. 1.1]{Bredon}) A {\em precosheaf} $\Q$ on $X$ is a {\em covariant} functor from the category of open subsets on $X$ to abelian groups. Thus, one has an abelian group $\Q(U)$ for each open subset $U$, and a morphism of groups $e_{VU}\colon \Q(V)\to \Q(U)$ for each $V\subseteq U$ such that $e_{UU}=\id$ for any $U$ and $e_{WU}=e_{VU}\circ e_{WV}$ for any $W\subseteq V\subseteq U$. A {\em cosheaf} is a precosheaf $\Q$ such that, for any open subset $U$ and any open covering $U=\underset i\cup U_i$, the sequence (with obvious morphisms)
\[ \underset{i,j}\oplus \Q(U_{ij})\to\underset i\oplus \Q(U_i)\to\Q(U)\to 0\] is exact, where $U_{ij}=U_i\cap U_j$. A morphism of  cosheaves $\Q_1\to\Q_2$ is just a morphism of functors, i.e., a natural transformation. We shall denote by $\Coshv(X)$ the category of cosheaves on $X$. 

A sequence of cosheaves $\Q_1\to\Q_2 \to\Q_3\to 0$ is called {\em exact} if $\Q_1(U)\to\Q_2(U)\to\Q_3(U)\to 0$ is exact   for any open subset $U$. In this case, the sequence $0\to \Hom(\Q_3,\Q)\to \Hom(\Q_2,\Q)\to \Hom(\Q_1,\Q)$ is exact for any $\Q$.

The direct sum $\oplus_i\Q_i$ of cosheaves is defined by $(\oplus_i \Q_i)(U)=\oplus_i \Q_i(U)$; it is a cosheaf, and it is the categorical direct sum.

\end{defn}

\begin{ejems}\label{ejems-cosheaves} (1) Any abelian group $G$ defines a  cosheaf $G^{\text{cos}}$, named constant cosheaf, by 
\[ G^{\text{cos}}(U)=G^{\oplus\pi_0(U)}\] (notice that we assume that $X$ is locally connected, and then the connected components are open subsets). For any cosheaf $\Q$ and any abelian group $G$ one has
\[ \Hom_{\Coshv}(\Q,G^{\text{cos}})=\Hom_\ZZ(Q(X),G).\]

(2) Let $j\colon U\hookrightarrow X$ be an open subset.  A cosheaf $\Q'$ on $X$ induces a cosheaf $\Q'_{\vert U}$ on each open subset $U$, defined as 
$$\Q'_{\vert U}(V)=\Q'(V)$$ for any open subset $V$ of $U$. This defines a functor
\[ \Coshv(X)\to\Coshv(U), \Q'\mapsto \Q'_{\vert U}\] that takes constant cosheaves into constant cosheaves (i.e., $(G^{\text{cos}})_{\vert U}=G^{\text{cos}}$). Conversely, a cosheaf $\Q$ on $U$ induces a cosheaf $j_!\Q$ on $X$, defined by $$(j_!\Q)(V)=Q(U\cap V).$$ This defines a functor
\[ j_!\colon \Coshv(U)\to\Coshv(X) \] and one has an adjunction
\[\Hom_{\Coshv(X)} (j_!\Q,\Q') = \Hom_{\Coshv(U)} (Q,\Q'_{\vert U}).\] For any abelian group $G$, we shall denote $G_U^{\text{cos}}:=j_!(G^{\text{cos}})$. If $V\subseteq U$, one has a natural morphism  $G_V^{\text{cos}}\to  G_U^{\text{cos}}$. 

(3) For any cosheaves $\Q,\Q'$, we can define a sheaf $\HHom(\Q,\Q')$ by:
\[ \HHom(\Q,\Q')(U)=\Hom(\Q_{\vert U},\Q'_{\vert U})\] and obvious restriction morphisms. We leave the reader to check that it is in fact a cosheaf.
\end{ejems}

If $F$ is a sheaf on $X$, the covariant behaviour of $\La (U,F_{\vert U})$ with respect to $U$ shows that there is a cosheaf attached to $F$:

\begin{defn} Let $F$ be a sheaf on $X$. The {\em associated cosheaf}, denoted by   $\co (F) $, is defined by
\[ \co (F)  (U):=\La(U,F_{\vert U})\] and for any $V\subseteq U$, the morphism $\co (F)(V)\to \co (F)(U)$ is $i_*\colon \La(V,F_{\vert V})\to \La(U,F_{\vert U})$, with $i\colon V\hookrightarrow U$. By Proposition \ref{F_U}, $\co (F)(U)=\La (X,F_U)$ and $\co (F)(V)\to \co (F)(U)$ is obtained by applying $\La(X,\quad)$ to $F_V\to F_U$.
\end{defn}

\begin{rem} $\co (F)$ is indeed a cosheaf. For any open subset $U$ and any open covering $U=\underset i\cup U_i$, one has an exact sequence of sheaves (let us denote $U_{ij}=U_i\cap U_j$)
\[ \underset{i,j}\oplus F_{U_{ij}}\to\underset i\oplus F_{U_i}\to F_U\to 0\] and then, applying $\La(X,\quad)$, an exact sequence of abelian groups
\[ \underset{i,j}\oplus \,\co (F)(U_{ij})\to\underset i\oplus \,\co (F)(U_i)\to \co (F)(U)\to 0.\]
\end{rem}

A morphism of sheaves $F_1\to F_2$ induces a morphism of cosheaves $\co (F_1)\to \co (F_2)$, thus we obtain a functor $$\co \colon \Shv(X)\to\Coshv(X).$$  

\begin{prop}\label{cos-properties}  The functor $\co$ satisfies:\begin{enumerate}
\item It commutes with restrictions: for any open subset $U$ and any sheaf $F$, one has: $\co (F)_{\vert U}=\co(F_{\vert U})$.
\item It commutes with $j_!$. For any open subset $j\colon U\hookrightarrow X$ and any sheaf $F$ on $U$ one has: $\co(j_!F)=j_!\,\co(F)$.
\item It takes constant sheaves into constant cosheaves: $\co(G)=G^{\text{cos}}$, for any abelian group $G$. Combining with {\rm (2)}, one has $$\co(G_U)=G_U^{\text{cos}}$$ for any abelian group $G$ and any open subset $U$.

\item It commutes with direct sums (more generally, with direct limits): $\co(\underset i\oplus  F_i)=\underset i \oplus \,\co(F_i)$.
\item It is right exact: if $F'\to F\to F''\to 0$ is an exact sequence of sheaves, then $\co(F')\to \co( F)\to  \co(F'')\to 0$ is an exact sequence of cosheaves.
\end{enumerate}
\end{prop}

\begin{proof} (1) For any open subset $V$ of $U$,
\[ [\co (F)_{\vert U}](V)= \co (F)(V) = \La (V,F_{\vert V})=\La(V,(F_{\vert U})_{\vert V})=\co(F_{\vert U})(V). \]

(2) For any open subset $V$ of $X$, one has, on the one hand
\[ (j_!\,\co F))(V)= \co(F)(U\cap V)=\La(U\cap V,F_{\vert U\cap V}),\] and, on the other hand (let us denote $j'\colon U\cap V\hookrightarrow V$)
\[\co(j_!F)(V)=\La(V,(j_!F)_{\vert V})=\La(V, (j')_! F_{\vert U\cap V})\overset{\ref{F_U}}=\La(U\cap V,F_{\vert U\cap V}).\]

(3) For any open subset $U$ one has
$$\cos(G)(U)=\La(U,G)=G^{\oplus\pi_0(U)}=G^{\text{cos}}(U). $$

(4) $\co(\underset i\oplus  F_i)(U)=\La(U,(\underset i\oplus  F_i)_{\vert U})= \La(U, \underset i\oplus  {F_i}_{\vert U}) = \underset i\oplus \La(U,  {F_i}_{\vert U}) =\underset i \oplus  \,\co(F_i)(U) = (\underset i\oplus\,  \co(F_i))(U).$ Analogous proof for the direct limit.

(5) Since $F\mapsto \La(U,F_{\vert U})$ is right exact, one concludes. 
\end{proof}

\begin{defn} Since the functor $\co\colon\Shv(X)\to\Coshv(X)$ is right exact and commutes with direct sums, it  has a right adjoint 
\[ \shf\colon \Coshv(X)\to\Shv(X).\] Thus, for any sheaf $F$ and cosheaf $\Q$ one has
\[\Hom_{\Coshv }(\co (F),\Q)=\Hom_{\Shv }(F,\shf(\Q)).\]
\end{defn}

The explicit description of $\shf(\Q)$ is given by the following:

\begin{prop} Let $\Q$ be a cosheaf. For any open subset $U$, one has
\[ \shf(\Q)(U)=\Hom_{\Coshv(X)}(\ZZ_U^{\text{\rm cos}},\Q)=\Hom_{\Coshv(U)}(\ZZ^{\text{\rm cos}},\Q_{\vert U})\] and for any $V\subseteq U$, the morphism $$\shf(\Q)(U)\to \shf(\Q)(V)$$ is obtained by taking $\Hom(\quad,\Q)$ in the natural morphism $\ZZ_V^\text{\rm cos}\to \ZZ_U^\text{\rm cos}$ (or it is the natural restriction morphism $\Hom(\ZZ^{\text{\rm cos}},\Q_{\vert U})\to \Hom(\ZZ^{\text{\rm cos}},\Q_{\vert V})$). In other words (see {\rm Examples \ref{ejems-cosheaves},   (3)}),
\[\shf(\Q)=\HHom(\ZZ^{\text{\rm cos}},\Q) .\]
\end{prop}

\begin{proof} One has $\shf(\Q)(U)=\Hom(\ZZ_U,\shf(\Q))=\Hom(\co(\ZZ_U),\Q)$. One concludes because $\co(\ZZ_U)=\ZZ_U^{\text{cos}}$, by  (3) of   Proposition \ref{cos-properties}. The statement about $V\subseteq U$ is immediate.
\end{proof}

%\begin{defn} Let $\Q$ be a cosheaf on $X$. The {\em associated sheaf}, denoted by $\shf(\Q)$, is defined by
%\[ \shf(\Q)(U):=\Hom_{\Coshv(X)}(\ZZ_U^\text{cos},\Q)\] and, for any 
%\end{defn}
%
%\begin{rem} $\shf(\Q)$ is a sheaf. For any open subset $U$ and any open covering $U=\underset i\cup U_i$ one has an exact sequence of cosheaves
%\[\underset{i,j}\oplus \ZZ^{\text{cos}}_{U_{ij}}\to \underset i\oplus \ZZ^{\text{cos}}_{U_i}\to \ZZ^{\text{cos}}_U\to 0\] and taking $\Hom(\quad,\Q)$ an exact sequence of abelian groups
%\[ 0\to \shf(\Q)(U)\to\underset i\prod \shf(\Q)(U_i)\to\underset {i,j}\prod \shf(\Q)(U_{ij}).\]
%\end{rem}

\begin{rem}
An alternative construction of $\shf(\Q)$ can be done in an analogous way to that of $\co (F)$ in the following way. We leave the reader to check the details. The functor $\Mod(\ZZ)\to \Coshv(X), G\mapsto G^\text{{cos}}$, has a right adjoint, because it is right exact and commutes with direct sums. Thus, for any cosheaf $\Q$ there is an abelian group $\Gamma(X,\Q)$ (the group of {\em sections} of $\Q$ on $X$) and an isomorphism
\[ \Hom(G,\Gamma(X,\Q))=\Hom (G^{\text{cos}},\Q)\] and it may be explicitely computed by the formula
\[ \Gamma(X,\Q)=\plim{i\in I}\Q(U_i)\] where $I$ is the set of connected and non empty open subsets of $X$ and $i\leq j$ if $U_i\subseteq U_j$. Now, for any open subsets $U\subseteq V$, the natural morphism  $\Hom (G^{\text{cos}},\Q_{\vert V}) \to \Hom (G^{\text{cos}},\Q_{\vert U}) $ induces a morphism
\[\Gamma(V,\Q_{\vert V})\to \Gamma(U,\Q_{\vert U}).\] The assignation $U\mapsto  \Gamma(U,\Q_{\vert U})$ is the sheaf $\shf(\Q)$ defined above; in other words, $$\shf(\Q)(U)=\Gamma(U,\Q_{\vert U}).$$
\end{rem}

\begin{prop}\label{shf-properties} The functor $\shf$ satisfies
\begin{enumerate} \item It commutes with restrictions: $\shf(\Q)_{\vert U}=\shf(\Q_{\vert U})$.

\item  It takes constant cosheaves into constant sheaves: $ \shf(G^{\text{\rm cos}})=G.$ 
\item It is left exact: If $0\to{\Q}'\to\Q\to\Q$ is an exact sequence of cosheaves (i.e., it is exact when applied to each $U$), then $0\to\shf(\Q')\to\shf(\Q) \to \shf(\Q)$ is an exact sequence of sheaves.
\end{enumerate}
\end{prop}

\begin{proof} (1) follows from adjunction and (2) of Proposition \ref{cos-properties}.

%\[\aligned (\shf(\Q)_{\vert U})(V)= \shf(\Q)(V) = \Hom_{\Coshv(X)}(\ZZ_V^{\text{cos}}, \Q)&=\Hom_{\Coshv(U)}(\ZZ_V^{\text{cos}}, \Q_{\vert U})\\ &=\shf(\Q_{\vert U})(V).\endaligned \]

(2)  For any open subset $U$ one has: \[ 
\shf(G^{\text{cos}})(U)= \Hom_{\Coshv(U)}(\ZZ^{\text{cos}}, G^{\text{cos}})   =\Hom_\ZZ( \ZZ^{\text{cos}}(U),G)   = \prod_{i\in\pi_0(U)}G =G(U). \]

(3) follows from adjunction.
\end{proof}

\begin{rem} The functors $\co \colon\Shv(X)\to\Coshv(X)$, $\shf\colon\Coshv(X)\to \Shv(X)$ do  not yield an equivalence between the categories of sheaves and cosheaves. However, they produce an equivalence between the categories of locally constant sheaves and locally constant cosheaves. Indeed, for any locally constant sheaf $L$ (resp. any locally constant cosheaf $\Lc$) the natural morphism $L\to \shf(\co(L))$ (resp., the natural morphism $\co(\shf(\Lc))\to\Lc$) is an isomorphism because it is so after restricting to any open subset $U$ such that $L_{\vert U}$ is constant (resp. such that $\Lc_{\vert U}$ is constant).
%In particular,
%\[\Hom_{\Shv}(G_U,H_V)=\Hom_{\Coshv}(G^{\text{cos}}_U,H^{\text{cos}}_V)\] for any abelian groups $G,H$ and any open subsets $U,V$.
\end{rem} 

Finally, to emphasize the symmetry between the functors $\co$ and $\shf$, we leave the reader to check the following:

\begin{rem} Let $j\colon U\hookrightarrow X$ be an open subset. For any sheaf $F$ and cosheaf $\Q$  one has:
\[ \aligned  \Hom_\ZZ(\co (F)(U),G)&=\Hom_{\Shv(X)}(F, j_*G)\\ \Hom_\ZZ(G,\shf(Q)(U))&=\Hom_{\Coshv(X)}(j_!G^{\text{cos}},Q)\endaligned.\]
\end{rem}
 
\bigskip
PROBLEM: Let $G$ be an abelian group and let us denote $H_0(X,G)=G^{\pi_0(X)}$. We want to derive this functor (by the left). We can take two directions. One is to think $H_0(X,G)$ as $G^{\text{cos}}(X)$, hence as the functor $\Q\mapsto \Q(X)$ on cosheaves applied to the constant cosheaf; then, we are lead to derive the functor $\Q\mapsto \Q(X)$. This direction finds serious difficulties related to the pathologies of the category of cosheaves. The main difficulty is that there is not a cosheafication procedure of a precosheaf (in an analogous way to the sheafication of a presheaf), and then, there is not a good definition of the  kernel (cosheaf) of a morphism of cosheaves; thus the category of cosheaves is not abelian. To overcome this difficulty one is lead to work with cosheaves of abelian pro-groups (see \cite{Prasolov}, \cite{Prasolov2}), and things become more difficult and technical. To our knowledge, there is still not a theory of homology of cosheaves avaialable.

A second direction is to think $H_0(X,G)$ as the cosections of the constant sheaf, i.e. $\La(X,G)$. This leads to derive the functor on sheaves $F\mapsto \La(X,F)$. The problem here is that the category $\Shv(X)$ has not enough projectives. If $X$ is a finite topological space (more generally, an Alexandroff space), the category $\Shv(X)$ has enough projectives and the left derived functor $\LL(X,\quad) $ of $\La(X,\quad)$ is left adjoint of $\pi_X^{-1}\colon D(\Abe)\to D(X)$; thus, the simplest idea is to define  (for a general topological space) $\LL(X,\quad) $ as the left adjoint of $\pi_X^{-1}$. The existence of this adjoint needs some additional local hypothesis on $X$, but do include, besides Alexandroff spaces,  the ``most common'' topological spaces .

\section{Homology}
 
\subsection{Notations}  For this section we shall use an extensive use of basic results on the (unbounded) derived category of sheaves on a topological space, which may be found in \cite{Spaltenstein}. Throughout the paper an injective or flat complex means a $K$-injective or $K$-flat complex in the sense of  \cite{Spaltenstein}.
 
 For any topological space $X$,  $D(X)$ denotes the  (unbounded) derived category of sheaves on $X$. If $X$ is a point, we obtain the derived category of abelian groups, denoted by $D(\ZZ)$. 
 For any $F\in D(X)$ and any open subset $U$, $\RR\Gamma(U,F)$ denotes the left derived functor of  sections on $U$.
 For any complexes of sheaves $F,F'$ on $X$, $\Hom^\pun(F,F')$ denotes the complex of homomorphisms and $\RR\Hom^\pun(F,F')$ its left derived functor. We denote by $\HHom^\pun(F,F')$ the complex of sheaves of homomorphisms and $\RR\HHom^\pun(F,F')$ its left derived functor. We denote $F\overset\LL\otimes F'$ the derived tensor product of $F$ and $F'$. One has natural isomorphisms
 \[ \aligned \RR\Hom^\pun(\ZZ_U,F)&=\RR\Gamma(U,F)\\ \RR\Gamma(U,\RR\HHom^\pun(F,F'))&=\RR\Hom^\pun(F_{\vert U}, F'_{\vert U})\\  \RR\Hom^\pun(F\overset\LL\otimes F', F'')&=\RR\Hom^\pun(F,\RR\HHom^\pun(F',F'')).\endaligned\]
 
 For any complex of sheaves $F$, we denote by $F^\vee$ its derived dual: $F^\vee:=\RR\HHom^\pun(F,\ZZ)$.  
 For any continuous map $f\colon X\to Y$, one has the  inverse image functor $f^{-1}\colon D(Y)\to D(X)$ and the (derived) direct image $\RR f_*\colon D(X)\to D(Y)$, which is right adjoint of $f^{-1}$.
 For any open subset $j\colon U\hookrightarrow X$ 
 one has the functor $j_!\colon D(U)\to D(X)$, which is left adjoint of $j^{-1}$, and we denote $F_U:=j_!j^{-1}F$. One has an exact triangle 
 \[ F_U\to F \to F_{X-U} \] where $F_{X-U}:=i_*i^{-1}F$, with $i\colon X-U\hookrightarrow X$.
 
 For any closed subset $Y$ of $X$, $\RR\Gamma_Y(X,F)$ denotes de right derived functor of  global sections supported on $Y$ and $\RR{\underline\Gamma}_YF$ its sheafified version. One has 
 \[\RR\Gamma_Y(X,F)=\RR\Hom^\pun(\ZZ_Y,F),\qquad \RR{\underline\Gamma}_Y F=\RR\HHom^\pun(\ZZ_Y,F)\]
 
 A perfect complex is a complex $\E$ which is locally isomorphic to a bounded complex 
 \[0\to G^{p}\to\cdots\to G^{q}\to 0\]
 of  constant sheaves, such that each $G^i$ is a finite and free $\ZZ$-module. In other words, for each point $x\in X$, there is a neighbourhood $U$ of $X$ and an isomorphism $\E_{\vert U}\simeq \pi_U^{-1}E$, where $E$ is a finite complex of finite and free $\ZZ$-modules. A complex $E$ of abelian groups is perfect if and only if is bounded (i.e. has bounded cohomology) and $H^i(E)$ is finitely generated for any $i$.  Finally, we shall use that an abelian group $G$ is zero if an only if $\Hom_\ZZ(G,\QQ/\ZZ)=0$ and a complex $E$ of abelian groups is $0$ (in the derived category) if and only if $E^\vee=0$ (see \cite[Ch.V, Prop. 13.7]{Bredon}); in particular, a morphism $E_1\to E_2$ in $D(\ZZ)$ is an isomorphism if and only if its dual $E_2^\vee\to E_1^\vee$ is so.

\subsection{Homology}

\begin{defn}\label{clc} We say that $X$ is {\em cohomologically trivial} if $H^i(X,G)=0$ for any $i>0$ and any abelian group $G$. Any homotopically trival space is cohomologically trivial. Any integral scheme is cohomologically trivial.  
\end{defn}

Let $D(\ZZ)$ be the derived category of abelian groups and $D(X)$ the derived category of sheaves  on $X$ (notice: we are considering the {\em unbounded} derived category, i.e., unbounded complexes). The category $D(X)$ has direct products, but it is not the direct product of complexes. Given a family of complexes $K_i$, we shall denote $\prod_iK_i$ the direct product complex and $\prod_i^{D}K_i$ the direct product in $D(X)$. If $K_i\to I_i$ are injective resolutions, then
$$ {\prod}_i^{D }K_i:= \prod_i I_i.$$

This pathology does not occur in $D(\ZZ)$; that is, for any collection $G_i$ of objects in $D(\ZZ)$, one has that $\prod_iG_i=\prod^D_iG_i$.
 
\begin{thm}\label{derivedproduct} Assume that $X$ is locally connected and locally cohomologically trivial.  The functor $\pi_X^{-1}\colon D(\ZZ)\to D(X)$ has a left adjoint
\[ \LL(X,\quad)\colon D(X)\to D(\ZZ).\] Thus, for any $F\in D(X)$ and any $G\in D(\ZZ)$ one has
\[ \Hom_{D(\ZZ)}(\LL(X,F),G)=\Hom_{D(X)}(F,\pi_X^{-1}G).\]
\end{thm}

\begin{proof} Let us denote $\pi=\pi_X$. By Brown representability for the dual (see \cite{Neeman2},\cite{Neeman3}, \cite{Krause}) it suffices to see that  $\pi^{-1}\colon D(\ZZ)\to D(X)$ commutes with direct products, i.e.:  For any collection   $\{ G_i\}_{i\in I}$, $G_i\in D(\ZZ)$, the natural morphism
\[\pi^{-1}(\prod_iG_i) \to {\prod}_i^{D }\pi^{-1}G_i\] is an isomorphism (in $D(X)$). Since any object  $G\in D(\ZZ)$ is isomorphic to $\prod_n H^n(G)[-n]$, we may assume that $G_i$ is an abelian group (in some degree). Now,  if $\pi^{-1}G_i\to I_i$ is an injective resolution of $\pi^{-1}G_i$, we have to prove that the morphism of complexes
\[ \prod_i\pi^{-1}G_i\to\prod_i I_i\] is a quasi-isomorphism. For any connected and cohomologically trivial open subset $U$
\[ \prod_i\Gamma(U,\pi^{-1}G_i)= \Gamma(U, \prod_i\pi^{-1}G_i)   \to \Gamma(U,\prod_i I_i)=\prod_i\Gamma(U,I_i)\] is a quasi-isomorphism because $\Gamma(U,\pi^{-1}G_i)\to \Gamma(U,I_i)$ is so.
\end{proof}

\begin{rem}\label{loc-compact} On locally compact Haussdorff spaces the locally cohomological trivial condition can be weakened by a {\em cohomologically locally connected} condition  (see \cite{Bredon})). This condition means that for each point $x\in X$ and neighbourhood $U$ of $x$, there is a neighbourhood $V\subset U$ of $x$ such that $H^n(U,\ZZ)\to H^n(V,\ZZ)$ is null for any $n>0$. Then, by universal coefficients theorem, one has that $H^n(U,G)\to H^n(V,G)$ is also null for any abelian group $G$. Hence, for any family of abelian groups $G_i$ one has that
\[ \ilim{x\in U}\prod_i H^n(U,G_i) =0\] for any $n>0$, and then (under locally connected hypothesis) one concludes that $\prod_i\pi^{-1}G_i\to\prod_i I_i$ is a quasi-isomorphism, and Theorem \ref{derivedproduct} still holds.
\end{rem}

\begin{rem} For the rest of the paper we shall assume that  the topological spaces involved satisfy that $\pi^{-1}$ has a left adjoint, i.e., that $\LL(\,\, ,\,\, )$ exists, since, for most of the results, the local hypothesis of Theorem \ref{derivedproduct} or Remark \ref{loc-compact} will not play any role. When they do, it will be explicitly stated.
\end{rem}

\begin{defn} For any complex $K$, we shall denote $H_i(K):=H^{-i}(K)$. For any $F\in D(X)$ we shall denote
\[ H_i(X,F):=H_i[\LL(X,F)]\] and call it the {\em $i$-th homology group with coefficients in $F$}.
\end{defn}

\begin{rem} Since $\pi_X^{-1}$ is exact, $\LL(X,\quad)$ is exact too (see \cite{Neeman2}). Thus, for any exact triangle
\[ F\to F'\to F'' \] one has an exact triangle
\[ \LL(X,F)\to \LL(X,F')\to \LL(X, F'') \] and then a long exact sequence of homology groups
\[ \cdots \to H_i(X,F)\to H_i(X,F')\to H_i(X,F'')\to  \cdots.\]
Moreover, $\LL(X,\quad)$ commutes with direct sums.
\end{rem}

\medskip 
\noindent{\bf\ \ Notation.} For any $G\in D(\ZZ)$, we shall often denote $G:=\pi_X^{-1}G$ the  complex of constant sheaves on $X$.
\medskip 

\begin{rem}\label{RHom-cosections} For any $F\in D(X)$ and any $G\in D(\ZZ)$ one has:
\[\RR\Hom^\pun(\LL(X,F),G)=\RR\Hom^\pun(F, G).\]
\end{rem}

\begin{proof}
Let $\epsilon\colon F\to \pi_X^{-1}\LL(X,F)$ be the unit map. The composition
\[ \RR\Hom^\pun(\LL(X,F),G)\to\RR\Hom^\pun(\pi_X^{-1}\LL(X,F),\pi_X^{-1}G)\overset{\epsilon^*}\to \RR\Hom^\pun(F,\pi_X^{-1}G)\] is an isomorphism, since it is so after taking $H^i$.  
\end{proof}

\begin{prop} For any sheaf $F$ on $X$ one has:

{\rm (a)} $H_i(X,F)=0$ for any $i<0$.

{\rm (b)} $H_0(X,F)=\La(X,F)$.

{\rm (c)}  $H_i(X,\quad)$ commutes with direct sums.

{\rm (d)} $\LL(X,\quad)$ maps $D^{\leq n}(X)$ into $D^{\leq n}(\ZZ)$ where $D^{\leq n}$ stands for complexes with $H^i=0$ for $i>n$.
\end{prop}

\begin{proof} (a) It suffices to see that $\Hom_\ZZ(H_i(X,F),\QQ/\ZZ)=0$ for $i<0$. Now $$\Hom_\ZZ(H_i(X,F),\QQ/\ZZ) = H^i\Hom^\pun(\LL(X,F),\QQ/\ZZ)= H^i\RR\Hom^\pun(F,\QQ/\ZZ)\overset{\ref{RHom-cosections}} = \Ext^i(F,\QQ/\ZZ),$$  which is $0$ for $i<0$. 

(b) Since $H_i(X,F)=0$ for $i>0$, for any abelian group $G$ one has  $$\aligned \Hom_\ZZ(H_0(X,F),G)& = H^0\RR\Hom^\pun(\LL(X,F),G)  =  H^0\RR\Hom^\pun(F,G)\\ & = \Hom_{\Shv}(F,G)=\Hom_\ZZ(\La(X,F),G).\endaligned$$

(c) This is immediate, since $\LL(X,\quad)$ commutes with direct sums.  

(d) It is an easy consequence of Remark \ref{RHom-cosections} and is left to the reader. We shall not make use of it.
\end{proof}

\begin{thm}[Duality between homology and cohomology]\label{homology-cohomology} For any $F\in D(X)$ and any $G\in D(\ZZ)$ one has:
\[\RR\Hom^\pun(\LL(X,F),G) =\RR\Gamma(X,\RR\HHom^\pun(F,G)),\]
 In particular,
\[ \RR\Hom^\pun(\LL(X,\ZZ),G) =\RR\Gamma(X,G)\] and
\[ \LL(X,F)^\vee=\RR\Gamma(X,F^\vee).\] where $(\underline\quad)^\vee$ stands for the derived dual $\RR\HHom^\pun(\underline\quad,\ZZ)$. Then one has the (split) exact sequence
\[ 0\to\Ext^1_\ZZ(H_{i-1}(X,F),\QQ/\ZZ)\to H^i(X,F^\vee)\to\Hom_\ZZ(H_i(X,F),\ZZ)\to 0.\]
\end{thm}

\begin{proof} In fact, $\RR\Hom^\pun(\LL(X,F),G)=\RR\Hom^\pun(F,G)=\Gamma(X, \RR\HHom^\pun(F,G)).$  If $F=\ZZ$, on obtains $\RR\Hom^\pun(\LL(X,\ZZ),G) =\RR\Gamma(X,\RR\HHom^\pun(\ZZ,G)) =  \RR\Gamma(X, G)$. The last formula is just to take $G=\ZZ$.
\end{proof}

\begin{coro}\label{coh-trivial} $X$ is cohomologically trivial if and only if $H_i(X,\ZZ)=0$ for any $i>0$.
\end{coro}

\begin{proof} $X$ is cohomologically trivial if and only if $\RR\Gamma(X,G)=\Gamma(X,G)$ for any abelian group  $G$. Since $\RR\Gamma(X,G)=\RR\Hom^\pun (\LL(X,\ZZ),G)$, this is equivalent to say that  $\LL(X,\ZZ)=\La(X,\ZZ)$, i.e., $H_i(X,\ZZ)=0$ for any $i>0$.
\end{proof}

\begin{prop}\label{inmersiones} Let  $j\colon U\hookrightarrow X$ be an open subset, $Y\hookrightarrow X$ a closed subset and $G\in D(\ZZ)$. 

\begin{enumerate}\label{inmcerrada}  \item For any $F\in D(U)$ one has
\[ \LL(U,F)=\LL(X,j_!F).\] In particular, for any $F\in D(X)$ one has
\[ \LL(U,F_{\vert U})=\LL(X,F_U).\]
\item For any $F\in D(X)$ one has
\[ \RR\Hom^\pun(\LL(Y,F_{\vert Y}),G)=\RR\Hom^\pun(F,G_Y).\]
\item Let $S=Y\cap U$. For any $F\in D(X)$ one has
\[\RR\Hom^\pun (\LL(S,F_{\vert S}),G)=\RR\Hom^\pun(F_U,G_Y).\]
\end{enumerate}
\end{prop}

\begin{proof} (1) The  functor $j_!\colon D(U)\to D(X)$  is  left adjoint of $j^{-1}$.  Since $\pi_U^{-1}=j^{-1}\circ \pi_X^{-1}$, we conclude that $\LL(U,\quad)=\LL(X,\quad)\circ j_!$.

(2) $\RR\Hom^\pun(\LL(Y,F_{\vert Y}),G)=\RR\Hom^\pun (F_{\vert Y}, G)=    
\RR\Hom^\pun(F,i_*G )$.

(3) By (2) applied to $S\hookrightarrow U$, one has $\RR\Hom^\pun (\LL(S,F_{\vert S}),G)=\RR\Hom^\pun(F_{\vert U} , (G_{\vert U})_S)$. Since $(G_{\vert U})_S= (G_Y)_{\vert U}$, we obtain $$\RR\Hom^\pun (\LL(S,F_{\vert S}),G)= \RR\Hom^\pun(F_{\vert U} , (G_Y)_{\vert U})=\RR\Hom^\pun(F_U,G_Y).$$
\end{proof}

%\begin{prop}\label{inmcerrada} For any closed subset $i\colon Y\hookrightarrow X$ one has
%\[ \RR\Hom^\pun(\LL(Y,F_{\vert Y}),G)=\RR\Hom^\pun(F,G_Y).\]
%\end{prop}
%
%\begin{prop} Let $S=U\cap Y$ a locally closed subespace. Then
%\[\RR\Hom^\pun (\LL(S,F_{\vert S}),G)=\RR\Hom^\pun(F_U,G_Y).\]
%\end{prop}

\begin{defn} Let $F$ be a complex of sheaves on $X$. We denote by $\La(X,F)$ the complex of abelian groups
\[\cdots\to\La(X,F^n)\to\La(X,F^{n+1})\to\cdots\] and one has a natural morphism of complexes $F\to \pi_X^{-1}\La (X,F)$, hence a morphism in $D(\ZZ)$
\[ \LL(X,F)\to \La(X,F).\]
An object $F\in D(X)$ is called {\em $\La$-acyclic}  if the  morphism $\LL(X,F)\to \La (X,F)$ is an isomorphism in $D(\ZZ)$. For example, a sheaf $F$ is $\La$-acyclic  iff $H_i(X,F)=0$ for any $i>0$. It is proved in a standard way that any bounded above complex of $\La$-acyclic sheaves is $\La$-acyclic.
\end{defn}

%\begin{thm} If $X$ is cohomologically trivial,  then $$H_i(X,\ZZ)=0$$ for any $i>0$.
%\end{thm}
%
%\begin{proof} It follows from the equality  $\RR\Hom^\pun(\LL(X,\ZZ),G)=  \RR\Gamma(X,G).$ 
%\end{proof}

\begin{prop}\label{basic-L-acyclic} An open subset $U$ is  cohomologically trivial if and only if $\ZZ_U$ is $\La$-acyclic.
\end{prop}

\begin{proof} Since $H_i(X,\ZZ_U)\overset{\ref{inmersiones}}=H_i(U,\ZZ)$, the result follows from Corollary \ref{coh-trivial}.
\end{proof}

\begin{prop}\label{covariance} Let $f\colon X\to Y$ a continuous map between locally connected spaces. For any $F\in D(Y)$ one has a natural morphism
\[ f_*\colon \LL(X,f^{-1}F)\to \LL(Y,F)\] hence a morphism
\[ f_*\colon H_i(X,f^{-1}F)\to H_i(Y,F).\]

For any continuous maps $  X\overset f\to Y\overset g\to Z$ one has: $(g\circ f)_*=g_*\circ f_*$. For any $X$, ${\Id_X}_*=\Id$.

\end{prop}

\begin{proof} The inverse image morphism $\Hom_{D(Y)}(F,G)\to\Hom_{D(X)}(f^{-1}F,G)$ gives, by Yoneda, the desired morphism. The equality $(g\circ f)^{-1}=f^{-1}\circ g^{-1}$ yields the equality $(g\circ f)_*=g_*\circ f_*$ and the equality $\Id_X^{-1}=\Id$ yields that ${\Id_X}_*=\Id$.
\end{proof}

\begin{rem}\label{rem-f_*} Let $F\in D(Y)$ and $\epsilon_F\colon F\to\pi_Y^{-1}\LL(Y,F)$ the unit morphism. The diagram
\[ \xymatrix{ f^{-1}F\ar[rr]^{f^{-1}(\epsilon_F)}\ar[rd]_{\epsilon_{f^{-1}F}} & & \pi_X^{-1}\LL(Y,F)
\\ & \pi_X^{-1}\LL(X,f^{-1}F)\ar[ru]_{\pi_X^{-1}(f_*)} & }\] is commutative.
\end{rem}

\begin{prop}[Mayer-Vietoris]\label{MV1} Let $X=U\cup V$, with $U,V$ open subsets. For any  $F\in D(X)$ one  has an exact triangle:
\[ \LL(U\cap V,F_{\vert U\cap V})\to \LL(U,F_{\vert U})\oplus \LL(V,F_{\vert V})\to  \LL(X,F)   \]
\end{prop}

\begin{proof} Applying $\LL(X,\quad)$ to the exact triangle $  F_{U\cap V}\to F_U\oplus F_V\to F$, one concludes by Proposition \ref{inmersiones}.
\end{proof}

\begin{prop}[Mayer-Vietoris2]\label{MV2} Let $X=Y\cup Z$, with $Y,Z$ closed subsets. For any  $F\in D(X)$ one  has a long exact sequence
\[\cdots \to   H_i(Y\cap Z,F_{\vert Y\cap Z}) \to H_i(Y,F_{\vert Y})\oplus  H_i(Z,F_{\vert Z})\to  H_i(X,F)\to  \cdots \]
\end{prop}

\begin{proof} For simplicity, we shall use the notation $\LL(C,F):=\LL(C,F_{\vert C})$ for any closed subset $C$ and any $F\in D(X)$. For any $G\in D(\ZZ)$ we have an   exact sequence
\[ 0\to G\to G_Y\oplus G_Z\to G_{Y\cap Z}\to 0.\] Applying $\RR\Hom^\pun(F,\quad)$ and Proposition \ref{inmcerrada} we obtain  an exact triangle
\[(*)\quad \RR\Hom^\pun(\LL(X,F),G)\to \RR\Hom^\pun(\LL(Y,F)\oplus \LL(Z,F),G)\to \RR\Hom^\pun(\LL(Y\cap Z,F),G)\]  % \RR\Hom(\LL(X,F),G)[1]\qquad \] 
hence morphisms
\[ \LL(X,F)[-1]\to \LL(Y\cap Z,F_{\vert Y\cap Z})\to \LL(Y,F_{\vert Y})\oplus \LL(Z,F_{\vert Z})\to \LL(X,F)\] and then a sequence 
\[\cdots \to H_{i+1}(X,F)\to H_i(Y\cap Z,F_{\vert Y\cap Z}) \to H_i(Y,F_{\vert Y})\oplus  H_i(Z,F_{\vert Z})\to  H_i(X,F)\to  \cdots \] which is exact because, taking $\Hom(\quad,\mathbb{Q}/\ZZ)$, one obtains the long exact sequence associated to the triangle $(*)$ (for $G=\QQ/\ZZ$).
\end{proof} 

\subsection{Vietoris and homotopy invariance}

\begin{thm}[See {\cite[Ch.V, Thm. 6.1]{Bredon}}]\label{h-inv} Let $f\colon X\to Y$ be a proper, separated and surjective morphism. Assume that $f^{-1}(y)$ is connected and cohomologically trivial for any $y\in Y$. Then, for any $F\in D(Y)$, the morphism
\[ f_*\colon \LL(X,f^{-1}F)\to \LL(Y,F)\] is an isomorphism.
% Let $\pi \colon X\times Y\to X$ be the natural projection, with $Y$ a compact, connected and cohomologically trivial space. For any sheaf $F$ on $X$, the map
%\[ \pi_*\colon \LL(X\times Y,\pi^{-1}F)\to \LL(X,F)\] is an isomorphism, hence
%\[ H_i(X\times Y,\pi^{-1}F)\overset\sim\to H_i(X,F).\]
\end{thm}

\begin{proof} Recall that $f_*$ is defined, by Yoneda, from the inverse image morphism $$\Hom_{D(Y)}(F,G)\to\Hom_{D(X)}(f^{-1}F,G)=\Hom_{D(Y)}(F,\RR f_*f^{-1}G).$$  Since  $G\to \RR f_*f^{-1}G$ is an isomorphism by the hypothesis and base change theorem (see \cite{Iversen}), we are done. 

% The natural morphism $F\to \RR f_*f^{-1}F$ is an isomorphism, by base change theorem. Hence
%\[\aligned \Hom_{D(\ZZ)}(\LL(X ,f^{-1}F), G) & =\Hom_{D(X )}(f^{-1}F,\pi_{X }^{-1}G) = \Hom_{D(X   )}(f^{-1}F,f^{-1}\pi_{Y}^{-1}G)\\ & =\Hom_{D(Y)}(F,\RR f_*f^{-1}\pi_{Y}^{-1}G)=\Hom_{D(Y)}(F,\pi_{Y}^{-1}G)\\ &=\Hom_{D(\ZZ)}(\LL(Y,F),G).\endaligned \]
%This proves the statement for $H_0$. Now, let us consider an exact sequence
%\[ 0\to F'\to \oplus \ZZ_U\to F\to 0\] with $U$ cohomologically trivial. Applying $\pi^{-1}$ one obtains
%\[ 0\to \pi^{-1}F'\to \oplus \ZZ_{U\times I}\to \pi^{-1}F\to 0.\]
%
%Since $\ZZ_U$ and $\ZZ_{U\times I}$ are coacyclic, one obtains that $H_1(X\times I,\pi^{-1}F)\to H_1(X,F)$ is an isomorphism for any $F$. By induction, taking into account that $H_i(X,F')=H_{i+1}(X,F)$ for $i>0$ (and the same happens in $X\times I$), one concludes.
\end{proof}

\begin{coro}[See {\cite[Ch. V, Cor. 6.13]{Bredon}}]\label{h-inv2} If $f,g\colon X\to Y$ are homotopic, then, for any $G\in D(\ZZ)$, the morphisms  $$f_*,g_*\colon\LL(X,G)\to\LL(Y,G)$$ coincide. Consequently, if $X$ and $Y$ are homotopic, then
\[ \LL(X,G)\simeq \LL(Y,G)\] for any  $G\in D(\ZZ)$.
\end{coro}

\begin{proof} By hypothesis, there exists a continuous map $H\colon X\times [0,1]\to Y$ such that $H_0=f$, $H_1=g$. It suffices to see that the morphism ${H_t}_*\colon \LL(X,G)\to\LL(Y,G)$ does not depend on $t$. Let $i_t\colon X\to X\times [0,1]$, $i_t(x)=(x,t)$. Since $H_t=H\circ i_t$, it suffices to see that ${i_t}_*\colon  \LL(X,G)\to\LL(X\times[0,1],G)$ does not depend on $t$. Let $\pi\colon X\times [0,1]\to X$ be the projection. Since ${i_t}_*\circ \pi_*=\id$ and $\pi_*$ is an isomorphism by Theorem \ref{h-inv}, we conclude that ${i_t}_*$ is the inverse of $\pi_*$, and then it does not depend on $t$.
\end{proof}

\subsection{Local homology and Excision}

\begin{defn} Let $Y$ be a closed subset of $X$. Let us denote by $\La^Y(X,F)$ the cokernel of the map 
\[ \La(X-Y,F_{\vert X-Y})\to \La(X,F).\]
Since $\La(X-Y,F_{\vert X-Y})=\La(X,F_{X-Y})$, it follows from the exact sequence
\[ 0\to F_{X-Y}\to F\to F_Y\to 0\] that $\La^Y(X,F)=\La(X,F_Y)$ and then 
\[ \Hom_\ZZ(\La^Y(X,F),G)=\Hom_{\Shv}(F_Y,G)=\left\{\aligned &\Hom_{\Shv}(F,{\underline\Gamma}_YG)\\ &\Gamma_Y(X,\HHom(F,G))\endaligned\right. .\]
\end{defn}

\begin{defn} For any $F\in D(X)$, we define $\LL^Y(X,F):=\LL(X,F_Y)$ and one has
\[ \RR\Hom^\pun(\LL^Y(X,F),G)=\RR\Hom^\pun(F_Y,G)=
\left\{\aligned &\RR\Hom^\pun(F,\RR{\underline\Gamma}_YG)\\ &\RR\Gamma_Y(X,\RR\HHom^\pun(F,G))\endaligned\right.
.\]
 We shall denote
\[ H_i^Y(X,F):=H_i[\LL^Y(X,F)].\]
\end{defn}

\begin{rem} The usual notation for this local homology would be $H_i(X,U;F)$; we have prefered to use a notation ``dual'' to that of local cohomology.
\end{rem}

\begin{rem}\label{localhomology} Applying $\LL(X,\quad)$ to the exact triangle $F_{X-Y}\to F\to F_Y$ one obtains the exact triangle
\[ \LL(X-Y,F_{\vert X-Y})    =\LL(X,F_{X-Y})\to \LL(X,F)\to \LL(X,F_Y)=\LL^Y(X,F)\] and then a  long exact sequence
\[ \cdots\to H_i(X-Y,F_{\vert X-Y})\to H_i(X,F)\to H_i^Y(X,F)\to  \cdots\]
\end{rem}
\begin{prop} For any closed subsets $Y,Z$ of $X$ and any $F\in D(X)$ one has an exact triangle
\[ \LL^{Y\cup Z}(X,F)\to \LL^Y(X,F)\oplus \LL^Z(X,F)\to \LL^{Y\cap Z}(X,F)\] and then a long exact sequence
\[ \cdots\to H_i^{Y\cup Z}(X,F )\to H_i^Y(X,F)\oplus H_i^Z(X,F)\to H_i^{Y\cap Z}(X,F)\to  \cdots\]
\end{prop}

\begin{proof}Apply $\LL(X,\quad)$ to the exact triangle $F_{Y\cup Z}\to F_Y\oplus F_Z\to F_{Y\cap Z}$.
\end{proof} 

\begin{rem} \label{pre-escision}For any open subset $j\colon U\hookrightarrow X$ and any $F\in D(U)$ one has:
\[ \LL^{Y\cap U}(U,F)=\LL^Y(X,j_!F).\]
\end{rem}

\begin{proof} $\LL^{Y\cap U}(U,F)=\LL(U,F_{Y\cap U})\overset{\ref{inmersiones}}=\LL(X,j_!F_{Y\cap U}) = \LL(X,(j_!F)_Y )=\LL^Y(X,j_!F)$.
\end{proof}

\begin{prop}[Excision]\label{excision} Let $U$ be an open subset of $X$ containing a closed subset $Y$ of $X$. The natural morphism
\[ \LL^Y(U,F_{\vert U})\to \LL^Y(X,F)\] is an isomorphism.
\end{prop}

\begin{proof} Since $\LL^Y(U,F_{\vert U}) \overset{ \ref{pre-escision}} = \LL^Y(X,F_U)$,  applying $\LL^Y(X,\quad)$ to the exact triangle $F_U\to F\to F_{X-U}$, it suffices to see that $\LL^Y(X,F_{X-U}) = 0$. But  $\LL^Y(X,F_{X-U})=\LL(X,(F_{X-U})_Y) = 0$ because $Y\cap (X-U)$ is empty.
\end{proof}

\subsection{Universal coefficients formula}

\begin{thm}[See {\cite[Ch.V, 3.10]{Bredon}}] \label{universalcoefficients} For any $F\in D(X)$ and any $G\in D(\ZZ)$, there is a natural isomorphism
\[   \LL(X,F\overset\LL \otimes G) \overset\sim\to \LL(X,F)\overset\LL\otimes G\]  In particular, for any abelian group $G$ one has a (split) exact sequence
\[ 0\to H_i(X,F)\otimes G\to H_i(X,F\overset\LL\otimes G)\to \Tor_1(H_{i-1}(X,F),G)\to 0.\]
\end{thm}

\begin{proof} Let us first define the morphism. Let $\pi\colon X\to\{*\}$ be the projection to a point. The unit map $F\to \pi^{-1}\LL(X,F)$ induces a morphism $F\overset\LL\otimes \pi^{-1}G \to \pi^{-1}\LL(X,F)\overset\LL\otimes \pi^{-1}G= \pi^{-1}(\LL(X,F)\overset\LL\otimes G)$, hence a morphism $\LL(X,F\overset\LL \otimes \pi^{-1}G) \to \LL(X,F)\overset\LL\otimes G$. In order to see that it is an isomorphism, we see the isomorphism after applying $ \Hom_{D(\ZZ)} (\quad, A)$ for any $A\in D(\ZZ)$. Now,
\[\aligned   \Hom_{D(\ZZ)} ( \LL(X,F)\overset\LL\otimes G,A)&=  \Hom_{D(\ZZ)} (\LL(X,F),\RR\Hom^\pun(G,A))\\ &=  \Hom_{D(X)} ( F ,\pi^{-1}\RR\Hom^\pun(G,A))\endaligned\] and \[\aligned
 \Hom_{D(\ZZ)} (  \LL(X,F\overset\LL \otimes \pi^{-1}G),A) &=   \Hom_{D(X)} (  F\overset\LL \otimes \pi^{-1}G,\pi^{-1}A) \\ & =  \Hom_{D(X)} (  F,\RR\HHom^\pun(\pi^{-1}G,\pi^{-1}A)
\endaligned\] and we conclude if we prove that 
\[ \pi^{-1}\RR\Hom^\pun(G,A) \to \RR\HHom^\pun(\pi^{-1}G,\pi^{-1}A)\] is an isomorphism. If $G$ is a free $\ZZ$-module, then the result follows because $\pi^{-1}$ commutes with direct products. If $G$ is a $\ZZ$-module, we conclude by putting $G$ as a cokernel of a monomorphism between free $\ZZ$-modules. Since any $G\in D(\ZZ)$ is isomorphic to $\oplus_n H^n(G)[-n]$, the result follows. 

\end{proof}

\subsection{Cap product}(See {\cite[Ch.V, Sec.10]{Bredon}})

\begin{thm}\label{cap} For any $F,F'\in D(X)$ there is a morphism
\[ \cup\colon \LL(X,F)\overset\LL\otimes \RR\Gamma(X,F')\to \LL(X,F\overset\LL\otimes F')\] and hence  morphisms
\[ \cup\colon H_p(X,F)\otimes H^q(X,F')\to H_{p-q}(X,F\overset\LL \otimes F').\]
\end{thm}

\begin{proof} The unit morphism $F\overset\LL\otimes F'\to\pi_X^{-1}\LL(X, F\overset\LL\otimes F')$ defines a morphism
\[ \Phi\colon F'\to \RR\HHom^\pun (F, \pi_X^{-1}\LL(X, F\overset\LL\otimes F')).\] Taking global sections, one obtains a morphism
\[ \Phi_\cup\colon \RR\Gamma(X,F')\to \RR\Hom^\pun (F, \pi_X^{-1}\LL(X, F\overset\LL\otimes F')) = \RR\Hom^\pun (\LL(X,F),  \LL(X, F\overset\LL\otimes F')),\] which is equivalent, by adjunction, to a morphism $$\cup\colon \LL(X,F)\overset\LL\otimes \RR\Gamma(X,F')\to \LL(X,F\overset\LL\otimes F').$$ 
%The morphism is obtained, by Yoneda, from the morphism
%\[ \Hom(F\overset\LL\otimes F',G)\to \Hom(\LL(X,F)\overset\LL\otimes \RR\Gamma(X,F'),G)= \Hom(\RR\Gamma(X,F'),\RR\Hom^\pun(F,G))\] which is the composition of the   isomorphism $$\Hom(F\overset\LL\otimes F',G)= \Hom(F,\RR\HHom^\pun(  F',G))$$ with the natural morphism (taking global sections) $$\Hom(F,\RR\HHom^\pun(  F',G))\to  \Hom(\RR\Gamma(X,F'),\RR\Hom^\pun(F,G)).$$
%
% Let $\pi\colon X\to\{*\}$ be the projection to a point. Then $\RR\Gamma(X,F')=\RR\pi_*F'$.
% The natural morphism $\pi^{-1}\RR\pi_*F'\to F'$ induces a morphism $F\overset\LL\otimes \pi^{-1}\RR\pi_*F'\to F \overset\LL\otimes F'$ and hence a morphism $\LL(X,F\overset\LL\otimes \pi^{-1}\RR\pi_*F')\to \LL(X, F\overset\LL\otimes F')$. Since $$\LL(X,F\overset\LL\otimes \pi^{-1}\RR\pi_*F')=\LL(X,F)\overset\LL\otimes \RR\pi_*F'$$ by the universal coefficients formula, we are done.
\end{proof}

\begin{thm}[See {\cite[Ch.V, Sec.10, (9)]{Bredon}}] Let $f\colon X\to Y$ be a continuous map. For any $F,F'\in D(Y)$ the diagram
\[\xymatrix{ \LL(Y,F) \ar[r]^{\Psi_{\cup_Y}\qquad\qquad} & \RR\Hom^\pun(\RR\Gamma(Y,F'),\LL(Y,F\overset\LL\otimes F'))
\\ \LL(X,f^{-1}F)\ar[u]^{f_*}\ar[r]^{\Psi_{\cup_X}\qquad\qquad\qquad} & \RR\Hom^\pun(\RR\Gamma(X,f^{-1}F'),\LL(X,f^{-1}(F\overset\LL\otimes F')))\ar[u]_{f_*^*}}
\] is commutative, where $f_*^*$ is the morphism induced by the morphisms $$f_*\colon \LL(X,f^{-1}(F\overset\LL\otimes F'))\to \LL(Y, F\overset\LL\otimes F'),\quad f^*\colon \RR\Gamma(Y,F')\to\RR\Gamma(X,f^{-1}F').$$ and $\Psi_{\cup_Y},\Psi_{\cup_X}$ are the morphisms corresponding to $\cup_Y,\cup_X$ by adjunction.  Thus, one has
\[ f_*(\alpha\cup f^*(\beta))=f_*(\alpha)\cup \beta\] for any $\alpha\in H_p(X,f^{-1}F)$, $\beta\in H^q(Y, F')$. 
\end{thm}

\begin{proof} The statement is a consequence of the functoriality of all the constructions. Let us give the details. First notice that the morphism $f^*\colon \RR\Gamma(Y,F')\to\RR\Gamma(X,f^{-1}F')$ is nothing but a particular case of the natural morphism $\RR\Hom^\pun(F,F')\to\RR\Hom^\pun  (f^{-1}F,f^{-1}F')$ (when $F=\ZZ$) that we also denote by $f^*$. Let us simplify the notations, by putting $\LL(F)=\LL(Y,F)$, $\RR\Gamma(F')=\RR\Gamma(Y,F')$ and so on.

 The morphism $\Phi_Y \colon F\to \RR\HHom^\pun(F',\pi_Y^{-1}\LL(F\overset\LL\otimes F'))$, induces a morphism $f^{-1} \Phi_Y \colon f^{-1}F\to \RR\HHom^\pun(f^{-1}F',\pi_X^{-1}\LL(F\overset\LL\otimes F'))$ and a commutative diagram
\[ \xymatrix{ \RR\Gamma(F')\ar[r]^{\Phi_{\cup_Y}\qquad\qquad}\ar[d]_{f^*} &  \RR\Hom^\pun(F,\pi_Y^{-1}\LL(F\overset\LL\otimes F'))\ar[d]^{f^*}
\\ \RR\Gamma(f^{-1}F') \ar[r]^{ \Gamma(f^{-1}\Phi_Y)\qquad\qquad} & \quad\RR\Hom^\pun(f^{-1}F,\pi_X^{-1}\LL (F\overset\LL\otimes F')).\quad
}\] On the other hand, the composition
\[ \RR\Gamma(f^{-1}F') \overset{\Gamma(f^{-1}\Phi_Y)}\longrightarrow     \RR\Hom^\pun(f^{-1}F,\pi_X^{-1}\LL (F\overset\LL\otimes F'))\overset{f_*}\to \RR\Hom^\pun(f^{-1}F,\pi_X^{-1}\LL (f^{-1}(F\overset\LL\otimes F')))\] is the morphism $\Phi_{\cup_X}$ (see Remark \ref{rem-f_*}).

The statement now follows from the adjunctions between $\Phi_{\cup_Y}$ and $\Psi_{\cup_Y}$ (and between $\Phi_{\cup_X}$ and $\Psi_{\cup_X}$).

%
%
%\[ f^{-1} F\to f^{-1}\RR\HHom^\pun(F',\pi_Y^{-1}\LL(Y,F\overset\LL\otimes F'))\to \RR\HHom^\pun(f^{-1}F',\pi_X^{-1}\LL(Y,F\overset\LL\otimes F'))\] and a commutative diagram
%\[ \xymatrix{ f^{-1} F\ar[r]\ar[rd]_{\Phi_{f^{-1}F,f^{-1}F'}}  & \RR\HHom^\pun(f^{-1}F',\pi_X^{-1}\LL(Y,F\overset\LL\otimes F'))\\    & \RR\HHom^\pun(f^{-1}F',\pi_X^{-1}\LL(X,f^{-1}(F\overset\LL\otimes F')))\ar[u]_{f_*}  }  \] and taking global sections a commutative diagram (see Remark \ref{rem-f_*})
%\[\xymatrix{ \RR\Gamma(Y,F)\ar[d]_{f^*}\ar[r] & \RR\Hom^\pun(F',\pi_Y^{-1}\LL(Y,F\overset\LL\otimes F'))\ar[rd]^{f^*} & 
%\\ \RR\Gamma(X, f^{-1}F)\ar[r] & \RR\Hom^\pun(f^{-1}F',\pi_X^{-1}\LL(X,f^{-1}(F\overset\LL\otimes F')))\ar[r]^{f_*} & \RR\Hom^\pun(f^{-1}F',\pi_X^{-1}\LL(Y,F\overset\LL\otimes F'))
%}\]
\end{proof}

\subsection{K\"{u}nneth formula}

Let $X_1,X_2$ be two topological spaces and $X_1\times X_2$ the product space.
For any $F_1\in D(X_1)$, $F_2\in D(X_2)$ let us denote
\[ F_1\boxtimes F_2:=(\pi_1^{-1}F_1)\overset\LL\otimes (\pi_2^{-1}F_2)\] where $\pi_i\colon X_1\times X_2\to X_i$ is the natural projection.

\begin{prop} For  any $F_1\in D(X_1)$, $F_2\in D(X_2)$ there is a natural morphism
\[ \LL(X_1\times X_2, F_1\boxtimes F_2)\to \LL(X_1,F_1)\overset\LL\otimes \LL(X_2,F_2).\]
\end{prop}

\begin{proof} The unit morphisms $F_i\to\pi_{X_i}^{-1}\LL(X_i,F_i)$ induce morphisms $$\pi_i^{-1} F_i\to \pi_i^{-1}\pi_{X_i}^{-1}\LL(X_i,F_i)=\pi_{X_1\times X_2}^{-1} \LL(X_i,F_i)$$  hence a morphism
\[ F_1\boxtimes F_2\to \pi_{X_1\times X_2}^{-1} \LL(X_1,F_1)\overset\LL\otimes \pi_{X_1\times X_2}^{-1} \LL(X_2,F_2) = \pi_{X_1\times X_2}^{-1} [\LL(X_1,F_1)\overset\LL\otimes \LL(X_2,F_2)]\] and then a morphism
\[  \LL(X_1\times X_2, F_1\boxtimes F_2)\to \LL(X_1,F_1)\overset\LL\otimes \LL(X_2,F_2) .\]
\end{proof}

Our aim now is to show that the K\"{u}nneth morphism is an isomorphism under local cohomological triviality assumptions (Theorem \ref{Kunneth}). For this, let us first construct a standard $\La$-acyclic resolution of any complex of sheaves.

\subsubsection{Standard resolution}\label{standard-res}

We assume that $X$ is locally connected and locally cohomologically trivial. For any $F\in D(X)$ we shall construct an $\La$-acyclic complex ${\mathcal C}_\pun F$ and a quasi-isomorphism ${\mathcal C}_\pun F\to F$, which are functorial on $F$. The construction mimics that of a flat resolution of  $F$. 

\begin{defn} Let $F$ be a sheaf on $X$.  We define $${\mathcal C}_0F:=\underset{U}\bigoplus\, \ZZ_U^{\oplus F(U)}$$ where $U$ runs over the set of connected and cohomologically trivial open subsets of $X$. \end{defn}

The equality $\Hom_{\Shv(X)}(\ZZ_U,F)=F(U)$ induces a natural morphism 
\[ {\mathcal C}_0F\to F\]  which is an epimorphism, because  the set of cohomologically trivial open subsets is a basis of $X$. Moreover, ${\mathcal C}_0F$ is $\La$-acyclic by Proposition \ref{basic-L-acyclic}. Now, to construct a resolution of $F$ by $\La$-acyclics one proceeds in the standard way:

Let $F_1$ be the kernel of ${\mathcal C}_0F\to F$.  We define ${\mathcal C}_1F:= {\mathcal C}_0F_1$ and  $F_2$ the kernel of ${\mathcal C}_0F_1\to F_1$. Recurrently, one defines ${\mathcal C}_iF:= {\mathcal C}_0F_i$ and $F_{i+1}$ the kernel of ${\mathcal C}_0F_i\to F_i$. One obtains a complex
\[  {\mathcal C}_\pun F= \cdots \to  {\mathcal C}_iF\to\cdots \to  {\mathcal C}_1F\to  {\mathcal C}_0F\] which is a resolution of $F$ by $\La$-acyclic sheaves, and then $ {\mathcal C}_\pun F$ is $L$-acyclic.

If $F$ is a bounded above complex of sheaves, we define ${\mathcal C}_\pun F$ as the simple complex associated to the bicomplex ${\mathcal C}_p F_q$. One has a quasi-isomorphism ${\mathcal C}_\pun F\to F$ and ${\mathcal C}_\pun F$ is $\La$-acyclic.

Finally, for a general (unbounded) complex $F$, let us consider the truncations $\tau_{\leq n}F$ and the morphisms of complexes $\phi_n\colon \tau_{\leq n}F\to \tau_{\leq n+1}F$. We define ${\mathcal C}_\pun F:=\ilim{n}\, {\mathcal C}_\pun (\tau_{\leq n}F)$. One has a quasi-isomorphism ${\mathcal C}_\pun F\to F$ . From the exact sequence
\[ 0\to \oplus_n {\mathcal C}_\pun (\tau_{\leq n}F)\overset{1-\phi_n} \longrightarrow \oplus_n {\mathcal C}_\pun (\tau_{\leq n}F)\to {\mathcal C}_\pun F\to 0\] one obtains an exact sequence ($\La(X,\quad)$ commutes with direct limits and direct sums) \[ 0\to \oplus_n \La (X,{\mathcal C}_\pun (\tau_{\leq n}F)) \to \oplus_n \La(X, {\mathcal C}_\pun (\tau_{\leq n}F))\to \La(X,{\mathcal C}_\pun F)\to 0\] and an exact triangle
\[      \oplus_n \LL(X,{\mathcal C}_\pun (\tau_{\leq n}F)) \to \oplus_n \LL(X, {\mathcal C}_\pun (\tau_{\leq n}F))\to \LL(X,{\mathcal C}_\pun F)  \] and a commutative diagram of exact triangles

\[ \xymatrix{ \oplus_n \La(X,{\mathcal C}_\pun (\tau_{\leq n}F))\ar[r] & \oplus_n \La(X,{\mathcal C}_\pun (\tau_{\leq n}F))\ar[r] & \La(X,{\mathcal C}_\pun F)\\ \oplus_n \LL(X,{\mathcal C}_\pun (\tau_{\leq n}F)) \ar[r]\ar[u]^{\wr} & \oplus_n \LL(X, {\mathcal C}_\pun (\tau_{\leq n}F))\ar[r]\ar[u]^{\wr} & \LL(X,{\mathcal C}_\pun F) \ar[u] 
}\] and we conclude that  ${\mathcal C}_\pun F$ is $\La$-acyclic. Thus, we have proved:

\begin{thm}\label{enough-L-acyclics} If $X$ is locally cohomologically trivial, the category of complexes of sheaves on $X$ has enough $\La$-acyclics. More precisely, for each complex of sheaves $F$ there are an $\La$-acyclic complex ${\mathcal C}_\pun F$ and a quasi-isomorphism ${\mathcal C}_\pun F\to F$ which are functorial in $F$. Hence
\[ \LL(X,F)\simeq \La (X, {\mathcal C}_\pun F).\] 
\end{thm}

\begin{rem} $\C_\pun F$ is also a flat resolution of $F$.
\end{rem}

\begin{thm}[See {\cite[Ch. V, Thm 13.1]{Bredon}}]\label{Kunneth} Assume that $X_1,X_2$ are locally cohomologically trivial topological spaces. If $U_1\times U_2$ is cohomologically trivial for any cohomological trivial open subsets $U_1,U_2$ of $X_1,X_2$ (it would suffice this condition on a basis of cohomological trivial open subsets of $X_1$ and $X_2$), then for  any $F_1\in D(X_1)$, $F_2\in D(X_2)$ the K\"{u}nneth morphism
\[ \LL(X_1\times X_2, F_1\boxtimes F_2)\to \LL(X_1,F_1)\overset\LL\otimes \LL(X_2,F_2)\]
is an isomorphism. Hence one has an exact sequence
%\[ 0\to \underset{p+q=n}\bigoplus H_p(X_1,F_1)\otimes H_q(X_2,F_2)\to H_n(X_1\times X_2, F_1\boxtimes F_2)\to \underset{p+q=n-1}\bigoplus \operatorname{Tor}_1(H_p(X_1,F_1), H_q(X_2,F_2))\to 0.\]
\[ 0\to \underset{p+q=n}\bigoplus H_p(F_1)\otimes H_q(F_2)\to H_n( F_1\boxtimes F_2)\to \underset{p+q=n-1}\bigoplus \operatorname{Tor}_1(H_p(F_1), H_q(F_2))\to 0,\] where we have used the abbreviated notation $H_i(F):=H_i(X,F)$.
\end{thm}

 \begin{proof} For any   open subsets $U_1,U_2$ of $X_1,X_2$, one has  $\ZZ_{U_1}\boxtimes \ZZ_{U_2}=\ZZ_{U_1\times U_2}$; hence, by hypothesis,  
$\ZZ_{U_1}\boxtimes \ZZ_{U_2}$ is $\La$-acyclic if $U_1,U_2$ are cohomologically trivial, and (assuming $U_i$ connected) 
\[\label{1}\tag{1} \La(X_1\times X_2,\ZZ_{U_1}\boxtimes \ZZ_{U_2})=\ZZ=\La(X_1,\ZZ_{U_1})\otimes \La(X_1,\ZZ_{U_1}).\]

Let $\C_\pun F_i$ be the  standard resolution of $F_i$ (see \ref{standard-res}). Recall that $\C_\pun F_i=\ilim{}  \C_\pun(\tau_{\leq n}F_i)$ and $\C_j(\tau_{\leq n}F_i)$ is a direct sum of sheaves of the form $\ZZ_{U_i}$, with $U_i$ cohomologically trivial. Then $\pi_i^{-1}\C_\pun F_i$ is a flat resolution of $\pi_i^{-1} F_i$ and then
\[ F_1\boxtimes F_2= (\pi_1^{-1}\C_\pun F_1)\otimes (\pi_2^{-1}\C_\pun F_2)=\ilim{} [\pi_1^{-1}\C_\pun(\tau_{\leq n}F_1)\otimes  \pi_2^{-1}\C_\pun(\tau_{\leq n}F_2)]. \] 
Since $\pi_1^{-1}\C_\pun(\tau_{\leq n}F_1)\otimes \pi_2^{-1} \C_\pun(\tau_{\leq n}F_2)$ is a bounded above  complex of sheaves that are direct sums of sheaves of the form $\ZZ_{U_1}\boxtimes \ZZ_{U_2}$, we conclude that it is $\La$-acyclic, and then $(\pi_1^{-1}\C_\pun F_1)\otimes (\pi_2^{-1}\C_\pun F_2)$ is also $\La$-acyclic. Then
\[ \LL(X_1\times X_2,F_1\boxtimes F_2)=\La(X_1\times X_2, (\pi_1^{-1}\C_\pun F_1)\otimes (\pi_2^{-1}\C_\pun F_2)).\] Since $\La(X_1\times X_2,\quad)$ commutes with direct limits and direct sums,  the equality $(1)$ yields an isomorphism
\[ \La(X_1\times X_2, (\pi_1^{-1}\C_\pun F_1)\otimes (\pi_2^{-1}\C_\pun F_2))=\La(X_1,\C_\pun F_1)\otimes \La(X_2,\C_\pun F_2)\] and we conclude because $\La(X_i,\C_\pun F_i)$ is a flat complex of abelian groups and $\LL(X_i,F_i)\simeq \La(X_i,\C_\pun F_i)$.

\subsubsection{Comparison with Borel-Moore homology} Let $X$ be a locally compact Haussdorff space. Assume that $X$ is locally connected and locally cohomologically trivial. For any sheaf $F$, let us denote $H_i^c(X,F)$ its Borel-Moore homology groups with compact support. All the following properties of Borel-Moore homology may be found in \cite{Bredon}. If $U$ is a connected and cohomologically trivial open subset, then $H^c_i(X,\ZZ_U)=0$ for $i>0$ and $H^c_0(X,\ZZ_U)=\ZZ$. Since Borel-Moore homology commutes with direct sums, we obtain that $H^c_i(X,\C_0F)=0$ for $i>0$ and $H^c_0(X,\C_0F)= H_0(X,\C_0F)=\La(X,\C_0F)$. Then $H^c_i(X,F)$ may be computed with the standard resolution $\C_\pun F$, i.e., $H^c_i(X,F)\simeq H_i[H^c_0(X,\C_\pun F)]\overset{\ref{enough-L-acyclics}}\simeq H_i(X,F)$.
%
% that $X_i$ is locally homotopically trivial. Then $F_i$ admits a resolution ${\mathcal C}_\pun F\to F$, where ${\mathcal C}_iF$ is a direct sum of sheaves of the form $\ZZ_U$ with $U$ homotopically trivial. 
%If $U_1,U_2$ are homotopically trivial open subsets of $X_1,X_2$, then $U_1\times U_2$ is a homotopically trivial open subset of $X_1\times X_2$ and ; thus,  Now, $\pi_i^{-1}{\mathcal C}_\pun F_i$ is a flat resolution of $\pi_i^{-1}F_i$ and $(\pi_1^{-1}{\mathcal C}_\pun F_1)\otimes (\pi_2^{-1}{\mathcal C}_\pun F_2)$ is a bounded above complex of $\La$-acyclic sheaves; hence
%\[\aligned  \LL(X_1\times X_2, F_1\boxtimes F_2)&= \La(X_1\times X_2, (\pi_1^{-1}{\mathcal C}_\pun F_1)\otimes (\pi_2^{-1}{\mathcal C}_\pun F_2))\\ & \overset{\eqref{1}} = \La(X_1,{\mathcal C}_\pun F_1)\otimes \La(X_2,{\mathcal C}_\pun F_2)
%\endaligned \]
\end{proof}

\subsection{Comparison with Poincar\'e-Verdier duality}\label{Sec-Poincare-Verdier}$\,$ Let us consider two types of topological spaces where one has  a Poincar\'e-Verdier duality theory:

 (A)  $X$ is Haussdorff, locally compact and locally of finite dimension. Let $\pi\colon X\to\{*\}$ the projection to a point. Then (see \cite{Spaltenstein}):

 The functor of cohomology with compact support $$\RR\Gamma_c(X,\quad)\colon D (X)\to D (\ZZ)$$ has a right adjoint $$\pi^!\colon D^+(\ZZ)\to D^+(X).$$  We shall denote $D_X=\pi^!\ZZ$, the dualizing complex. For any open subset $U$, one has $(D_X)_{\vert U}=D_U$. For any $G\in D(\ZZ)$ there is a morphism
  \[ \pi^{-1}G\overset\LL\otimes D_X\to\pi^{!} G.\] Indeed, applying $\RR\Hom^\pun(\quad,G)$ to the unit morphism $\RR\Gamma_c(X,D_X)\to\ZZ$, gives a morphism
\[ G\to \RR\Hom^\pun(\RR\Gamma_c(X,D_X),G)=\RR\Hom^\pun(D_X,\pi^!G)=\RR\pi_*\RR\HHom^\pun (D_X,\pi^!G)\] and then a morphism $\pi^{-1} G\to \RR\HHom^\pun (D_X,\pi^!G)$, i.e., a morphism $\pi^{-1}G\overset\LL\otimes D_X\to\pi^{!} G$.

(B) $X$ is a finite topological space. The functor $$ \RR\Gamma (X,\quad)\colon D (X)\to D (\ZZ)$$ has a right adjoint (see \cite{Navarro}, or \cite{Sanchoetal}), that we also denote by $\pi^!$ (and then $D_X=\pi^!\ZZ$).   An important difference with (A) is that, for an open subset $U$, $(D_X)_{\vert U}$ needs not coincide with $D_U$. As in (A), one has a natural morphism $\pi^{-1}G\overset\LL\otimes D_X\to\pi^{!} G$ (by the same arguments).

In order to treat simultaneously both cases let us consider the following definition.

\begin{defn} A {\em duality theory} on $X$ is an exact functor $\omega\colon D(X)\to D(\ZZ)$ that has a right adjoint $\omega^!\colon D(\ZZ)\to D(X)$. We shall denote $D_X^\omega=\omega^!(\ZZ)$. For each $G\in D(\ZZ)$, one has a natural morphism
\[ \tag{1}\label{dual}\pi^{-1}G\overset\LL\otimes D_X^\omega\to \omega^!G.\] For each $F\in D(X)$, we shall denote
\[ H^i_\omega(X,F):=H^i(\omega(F)).\]

\begin{rem} If we take $\omega=\LL(X,\quad)$, then $\omega^!=\pi^{-1}$, the morphism $\eqref{dual}$ is the identity and $H^i_\omega(X,F)=H_{-i}(X,F)$.
\end{rem}

If $j\colon U\hookrightarrow X$ is an open subset of $X$, we shall denote $\omega_{\vert U}=\omega\circ j_!$, whose right adjoint is $j^{-1}\circ \omega^!$. Then $D_U^{\omega_{\vert U}}=(D_X^\omega)_{\vert U}$. In case (A), $\omega=\RR\Gamma_c(X,\quad)$ and $\omega_{\vert U}=\RR\Gamma_c(U,\quad)$. However, in case (B), $\omega=\RR\Gamma(X,\quad)$ but $\omega_{\vert U}$ needs not agree with $\RR\Gamma(U,\quad)$. Thus, in case (B) we have two different duality theories on $U$: $\RR\Gamma(U,\quad)$ and $\RR\Gamma(X,\quad)_{\vert U}$.
\end{defn}

\begin{prop}\label{PV-comparison} Let $(X,\omega)$ be a topological space with a duality theory. For any $F\in D(X)$ there is a natural morphism
$$\Phi_F\colon \omega( F\overset\LL\otimes D^\omega_X)\to \LL(X,F) $$ hence a morphism
\[ H^{-i}_\omega (X,F\overset\LL\otimes D^\omega_X)\to H_i(X,F).\]
\end{prop}
\begin{proof}  The unit morphism $\epsilon\colon F\to \pi^{-1}\LL(X,F)$ induces a morphism $$F\overset\LL\otimes D^\omega_X\overset{\epsilon\otimes 1}\to \pi^{-1}\LL(X,F)\overset\LL\otimes D^\omega_X \overset{\eqref{dual}}\to \omega^!\,\LL(X,F)$$ and then a morphism
\[ \omega( F\overset\LL\otimes D^\omega_X)\to \LL(X,F).\]
\end{proof}

% For any $G\in D^+(X)$, one has natural morphisms
%\[\aligned \Hom_{D(\ZZ)}&(\LL(X,F),G) =\Hom_{D(X)}(F,\pi_X^{-1}G)\to \Hom_{D(X)}(F\overset\LL\otimes D_X, \pi_X^{-1}G\overset\LL\otimes D_X)\\ & \to 
%\Hom_{D(X)}(F\overset\LL\otimes D_X, \pi_X^{!}G)=\Hom_{D(\ZZ)}(\RR\Gamma_c(X, F\overset\LL\otimes D_X), G)\endaligned\]

The characterization of those topological spaces $(X,\omega)$ for which $H^{-i}_\omega (X,F\overset\LL\otimes D^\omega_X)\to H_i(X,F)$ is an isomorphism for any $F$ is given in the following:

\begin{thm}[See {\cite[Ch.V, Thm. 9.2]{Bredon}}]\label{Poincare-Verdier} Let $(X,\omega)$ be a topological space with a duality theory. The following conditions are equivalent
\begin{enumerate}
\item $H^{-i}_\omega (X,F\overset\LL\otimes D^\omega_X)\to H_i(X,F)$ is an isomorphism for any $F\in D(X)$ and any $i$.
\item There is a basis $\frak B$ of open subsets such that $H^{-i}_\omega(X,\ZZ_U\otimes D_X^\omega)\to H_i(U,\ZZ)$ is an isomorphism for any $U\in\mathfrak B$ and any $i$. 
\item The natural morphism $\ZZ\to\RR\HHom^\pun(D^\omega_X,D^\omega_X)$ is an isomorphism.
%\item The natural morphism $\RR\Gamma_c(U,D_U)\to\ZZ$ is an isomorphism for a basis of   open subsets $U$ of $X$.
%\item The natural morphism $\RR\Gamma_c(U,D_U)\to\ZZ$ is an isomorphism for a basis of  open subsets $U$ of $X$.
\end{enumerate}
{\rm In any of these cases, we shall say that $(X,\omega)$ is a {\em $PV$-space} (a Poincar\'e-Verdier space).}
\end{thm}

\begin{proof} Notice that
\[ \omega(F\overset\LL\otimes D_X^\omega)^\vee=\RR\Hom^\pun(F\overset\LL\otimes D_X^\omega,D_X^\omega)=\RR\Hom^\pun(F,\RR\HHom^\pun(D_X^\omega,D_X^\omega))\] and then 
\[ \omega(\ZZ_U \otimes D_X^\omega)^\vee= \RR\Gamma(U,\RR\HHom^\pun(D_X^\omega,D_X^\omega)).\]

Then, $\Phi_F\colon \omega(F\overset\LL\otimes D_X^\omega)\to \LL(X,F)$ is an isomorphism for any $F$ if and only if $$\Phi_F^\vee\colon \RR\Hom^\pun(F,\ZZ)\to \RR\Hom^\pun(F,\RR\HHom^\pun(D_X^\omega,D_X^\omega))$$ is an isomorphism for any $F$, i.e., iff $\ZZ\to \RR\HHom^\pun(D_X^\omega,D_X^\omega))$ is an isomorphism. This gives the equivalence of (1) and (3). Analogously, $\Phi_U\colon \omega(\ZZ_U \otimes D_X^\omega)\to\LL(U,\ZZ)$ is an isomorphism for any $U\in\mathfrak B$ if and only if $\Phi_U^\vee\colon \RR\Gamma(U,\ZZ)\to\RR\Gamma(U,  \RR\HHom^\pun(D_X^\omega,D_X^\omega))$ is an isomorphism for any $U\in\mathfrak B$, i.e., iff 
$\ZZ\to \RR\HHom^\pun(D_X^\omega,D_X^\omega))$ is an isomorphism.
% Let us see first that (1) and (3) are equivalent. The morphism $\RR\Gamma_c(X, F\overset\LL\otimes D_X)\to \LL(X,F)$ is an isomorphism if and only if it is so after applying $\RR\Hom^\pun(\quad,\ZZ)$ (see \cite[Prop. 13.7]{}). Now
% \[ \RR\Hom^\pun(\RR\Gamma_c(X, F\overset\LL\otimes D_X),\ZZ)= \RR\Hom^\pun (F\overset\LL\otimes D_X ,D_X) = \RR\Hom^\pun(F,\RR\HHom^\pun(D_X,D_X))\] and 
% \[   \RR\Hom^\pun(\LL(X, F),\ZZ)=\RR\Hom^\pun(F,\ZZ).\] Hence, $\RR\Gamma_c(X, F\overset\LL\otimes D_X)\to \LL(X,F)$ is an isomorphism for any $F$ if and only if $\ZZ\to\RR\HHom^\pun(D_X,D_X)$ is an isomorphism.
%
% (1) $\Rightarrow$ (2) is clear and (2) $\Rightarrow$ (4) by taking $F=\ZZ_U$ with $U$ connected and cohomologically trivial. Finally, let us see (4) $\Rightarrow$ (3). It suffices to see that  $\ZZ\to\RR\HHom^\pun(D_X,D_X)$ is an isomorphism at the stalk of any point $x\in X$. If $U$ is a neighbourhood of $x$ belonging to the basis, one has 
% \[ \RR\Gamma(U,\RR\HHom^\pun(D_X,D_X))= \RR\Hom^\pun(D_U,D_U)=\RR\Hom^\pun(\RR\Gamma_c(U,D_U),\ZZ)\simeq \ZZ,\] and then $\RR\HHom^\pun(D_X,D_X)_x\simeq \ZZ$. 
\end{proof}

\begin{rem}
(1) Any finite dimensional topological manifold $X$ (with or without boundary) is a $PV$-space, for $\omega=\RR\Gamma_c(X,\quad)$, since it satisfies condition (3) of Theorem \ref{Poincare-Verdier}.

(2) Let $(X,\omega)$ be a topological space with a duality theory and $U$ an open subset. If $(X,\omega)$ is a $PV$-space, so is $(U,\omega_{\vert U})$. Conversely, if $(U_i,\omega_{\vert U_i})$ is a $PV$-space for an open covering $X=\cup_i U_i$, then $(X,\omega)$ is a $PV$-space.

(3) Any topological space $X$ is a $PV$-space with respect to $\omega=\LL(X,\quad)$.
\end{rem}

If $Y$ is the boundary of a topologial manifold $X$ with boundary, then $\RR{\underline\Gamma}_Y\ZZ=0$. This motivates the following:

\begin{coro} Let $(X,\omega)$ be a topological space with a duality theory. Assume that $D_X^\omega$ is supported on an open subset $j\colon U\hookrightarrow X$ (i.e., $D_X=j_!j^{-1} D^\omega_X$) and that $(U,\omega_{\vert U})$ is a $PV$-space.     The following conditions are equivalent (let us denote $Y=X-U$):
\begin{enumerate}
\item $(X,\omega)$ is a $PV$-space.
\item $\RR{\underline\Gamma}_Y\ZZ=0$.
\item  For any $F\in D(X)$ and any $i$, $H_i^Y(X,F)=0$.
\end{enumerate}
\end{coro}

\begin{proof} Let us denote $\omega'=\omega_{\vert U}$. By hypothesis, $D_X^\omega=j_!j^{-1}D_X^\omega = j_!D_U^{\omega'}$ and $\ZZ=\RR\HHom^\pun(D_U^{\omega'},D_U^{\omega'})$. Then
\[ \RR\HHom^\pun(D_X^\omega,D_X^\omega)= \RR\HHom^\pun(j_!D_U^{\omega'},j_!D_U^{\omega'})=\RR j_* \RR\HHom^\pun( D_U^{\omega'}, D_U^{\omega'})=\RR j_*\ZZ.\]
Moreover one has the exact triangle
\[ \RR{\underline\Gamma}_Y\ZZ\to\ZZ\to \RR j_*\ZZ.\] Hence, $X$ is a $PV$-space $\Leftrightarrow$ $\ZZ\to \RR j_*\ZZ$ is an isomorphism $\Leftrightarrow$ $ \RR{\underline\Gamma}_Y\ZZ=0$.

The equivalence of (2) and (3)follows from the equality $\LL^Y(X,F)^\vee=\RR\Hom^\pun(F,\RR{\underline\Gamma}_Y\ZZ)$.
\end{proof}

\subsubsection{Homological manifolds}\label{homologicalmanifolds}
In this section we assume that we are in case (A) or (B). We leave to the reader to give the natural generalizations to a topological space $X$ with a duality theory $\omega$.

\begin{defn} (See \cite[Ch.V, Def. 9.1]{Bredon}) Let $X$ be a topological space of type (A) or (B). We say that $X$ is an {\em $n$-dimensional  homological manifold} if $D_X \simeq \TT_X[n]$, where $\TT_X$ is a sheaf on $X$ locally isomorphic to the constant sheaf $\ZZ$. 
%More generally, let $Y$ be a closed subset of $X$; we say that $X$ is an {\em $n$-dimensional  homological manifold with boundary $Y$}, if $\RR{\underline\Gamma}_{Y}\ZZ=0$ and $D_X \simeq \TT_X[n]$, where $\TT_X$ is a sheaf on $X$ supported on  $X-Y$ and ${\TT_X}_{\vert X-Y}$ is   locally isomorphic to the constant sheaf $\ZZ$. 
\end{defn}

\begin{rem} Any $n$-dimensional  homological manifold   is a $PV$-space, with respect to $\omega=\RR\Gamma_c(X,\quad)$ in case (A) (resp. $\omega=\RR\Gamma (X,\quad)$ in case (B)).
\end{rem}

%\begin{prop} Let $X$ be an $n$-dimensional homological manifold with boundary $Y$. For any $F\in D(X)$ one has
%\[ \LL(X,F_Y)=0.\] In other words, $H_i^Y(X,F)=0$ for any $i$.
%\end{prop}
%
%\begin{proof} One has $\RR\Hom^\pun(\LL(X,F_Y),\ZZ)=\RR\Hom^\pun(F,\RR{\underline\Gamma}_Y\ZZ)$. 
%%We conclude if we prove that $\RR{\underline\Gamma}_YG=0$ for any $G\in D(\ZZ)$. Let $y\in Y$ and let $V$ be a neighbourhood of $y$ homemomorphic to a half ball (whose boundary is $Y\cap V$). If $U$ is the interior of $V$ (which is homemomorphic to a ball), then $\RR\Gamma(V,G)\to\RR\Gamma( U,G)$ is an isomorphism (both are isomorphic to $G$), and then $\RR\Gamma_{Y\cap V}(V,G)=0$. This proves that $\RR{\underline\Gamma}_YG=0$.
%\end{proof}

\begin{thm} [See {\cite[Ch.V, Thm. 15.1]{Bredon}}]  Let $X$ be a connected topological of type {\rm (A)} or {\rm (B)}.  Then $X$ is an $n$-dimensional  homological manifold   (for some $n$) if and only if $D_X$ is perfect.
\end{thm}

\begin{proof} The direct is immediate. For the converse, it suffices to see that for any closed point $p\in X$ there is a neighbourhood $U\ni p$ such that $(D_X)_{\vert U}\simeq\ZZ[n]$. First notice that
\[ \RR\Gamma_p(X,D_X)=\RR\Hom^\pun(\ZZ_{\{p\}},D_X)=\RR\Hom^\pun (\RR\Gamma (p,\ZZ),\ZZ)=\ZZ,\] hence $H^\pun_p(X,D_X)=\ZZ$. Now, by hypothesis, there is a neighbourhood $U$ of $x$ such that $(D_X)_{\vert U}\simeq \pi_U^{-1} D$, where $\pi_U\colon U\to \{*\}$ is the projection to a point and $D$ is a bounded complex of free $\ZZ$-modules of finite rank.  Then
\[ \ZZ=H^\pun_p(X,D_X)=H^\pun_p(U,(D_X)_{\vert U})=H^\pun_p(U,\ZZ)\otimes D.\] This implies (it is an exercise on complexes of abelian groups) that, for some integer $n$, $H^\pun_p(U,\ZZ)\simeq \ZZ[-n]$ and $D\simeq \ZZ[n]$. Then $(D_X)_{\vert U}\simeq\ZZ[n]$.
\end{proof}

\section{Homology of sheaves and homology of groups}\label{homologyofgroups}

In this section we make an additional hypothesis on the topological space: we assume that $X$ is connected and locally simply connected (by simply connected we mean a topological space such that every covering is trivial). Let $$\phi\colon \wt X\to X$$ be a universal cover and $$G=\Aut_X\wt X$$ the fundamental group. 

Let us briefly summarize the content of this section. Let us denote $\LConst(X)$ the category of locally constant sheaves on $X$ and $D(\LConst(X))$ its derived category.  Let us consider the inclusion functor $i\colon D(\LConst(X))\to D(X)$ and the commutative diagram 
\[\xymatrix{D(\ZZ)\ar[rr]^{\pi^{-1}}\ar[rd]_{\pi^{-1}} & & D(\LConst(X))\ar[ld]_i \\ & D(X) &
}\] Homology of groups is nothing but the left adjoint of $\pi^{-1}\colon D(\ZZ)\to D(\LConst(X))$, taking into account  the equivalence $D(\LConst(X))=D(G)$, where $D(G)$ is the derived category of the category of $G$-modules. Homology of sheaves is, by definition, the left adjoint of  $\pi^{-1}\colon D(\ZZ)\to D(X)$. We shall see that $i$ has also a left adjoint $\LLc\colon D(X)\to D(\LConst(X))$. Hence, homology of sheaves is the composition of $\LLc$ with homology of groups.  We shall compute the functor $\LLc$ in terms of homology of sheaves in the universal covering.

\subsection{}
Let us begin with the underived version of the above situation. Let us denote   $\Mod( G )$ the category of (left) $\ZZ[G]$-modules, i.e., the category of $G$-modules. For any abelian group $A$, we denote by $A_{\rm tr}$ the trivial $G$-module (i.e., it is the $\ZZ$-module $A$ with the trivial action of $G$). For any $G$-module $M$, we denote by $M_G$  the quotient module of coinvariants:  $$M_G=\ZZ\otimes_{\ZZ[G]}M.$$  
Then we have functors
\[\aligned  \Tr\colon \Mod(\ZZ)&\to\Mod(G),\\ A&\mapsto A_\tr  \endaligned \qquad \aligned(\quad)_G\colon \Mod(G)&\to \Mod(\ZZ)\\ M&\mapsto M_G\endaligned\] and $(\quad)_G$ is left adjoint of $\Tr$. 

One has a well known equivalence between the category of locally constant sheaves on $X$ and the category of $G$-modules; for each locally constant sheaf $\M$ on $X$, we shall denote by $M$ the corresponding $G$-module and viceversa. Thus
\[\tag{*}\label{LC=G-mod}\aligned \LConst(X)&=\Mod( G )\\ \M&\leftrightarrow M\endaligned\] that interchanges constant sheaves with trivial $G$-modules: $A\leftrightarrow A_{\rm tr}$. The assignation $\M\mapsto M$ may be described in terms of  cosections in the following way: $M=\La(\wt X,\phi^{-1}\M)$, which is a $G$-module because $\phi^{-1}\M$ is a $G$-sheaf on $\wt X$.

 Via the equivalence \eqref{LC=G-mod}, the functor $\Tr$ coincides with $\pi^{-1}$ and then, by adjunction, $(\quad)_G$ coincides with $\La(X,\quad)$; that is
\[ \tag{1}\label{cos=coinv} \La(X,\M)=M_G.\]

Let us now consider the inclusion functor $\LConst(X)\hookrightarrow \Shv(X)$. 

\begin{prop} The inclusion functor $ \LConst(X)\hookrightarrow \Shv(X)$ has a left adjoint
$$\lc \colon \Shv(X)\to\LConst(X).$$  Thus, for any sheaf $F$  and any locally constant sheaf $\M$ one has
\[ \Hom (F,\M)=\Hom(\lc(F),\M).\]
\end{prop} 
\begin{proof} Since the inclusion functor is exact, it suffices to see that it commutes with direct products. Indeed, if $\{\M_i\}_{i\in I}$ is a collection of locally constant sheaves, we have to see that the direct product sheaf $\prod_i\M_i$ is also locally constant. Restricting to a simply connected open subset, we are reduced to prove that the direct product of constant sheaves is constant. This is given by Proposition \ref{directprodconstant}.
\end{proof}

%\begin{ejem} For any simply connected open subset $U$ one has $\lc(\ZZ_U)=\ZZ$.
%\end{ejem}
 
 For any sheaf $F$ one has a natural morphism  $F\to \lc(F)$, and for any locally constant sheaf $\M$ a natural isomorphism $\lc(\M)\to\M$.
It is direct from the representability that $\lc$ is right exact and commutes with direct sums (more generally, with direct limits). 

Via the equivalence $\LConst(X)=\Mod(G)$, we can see $\lc$ as a functor  $\lc\colon \Shv(X)\to \Mod(G)$, whose description in terms of the universal cover is given by the following:

\begin{prop}\label{lc=cosections} The $G$-module corresponding to $\lc(F)$  is $\La(\wt X,\phi^{-1}F)$ (which is a $G$-module because $\phi^{-1}F$ is a $G$-sheaf on $\wt X$). Thus, we shall put
\[ \lc(F)=\La(\wt X,\phi^{-1}F),\] via the equivalence \eqref{LC=G-mod}. In particular, for any simply connected open subset $U$ of $X$ one has
\[ \lc(\ZZ_U)=\ZZ[G]\] as $G$-modules.
\end{prop}

\begin{proof} This is essentially a consequence of the Galois correspondence between sheaves on $X$ and $G$-sheaves on $\wt X$, that we recall now.  If $F$ is a sheaf on $X$, then $\phi^{-1}F$ is a $G$-sheaf on $\wt X$, since its \'etale space  $\wt{\phi^{-1}F}=\wt X\times_X\wt F$ is endowed with the action of $G$ given by the action on $\wt X$ and the trivial action on $\wt F$. One has then a Galois equivalence
\[ \aligned \Shv(X)&\to  \GShv(\wt X)\\ F&\mapsto \phi^{-1}F\endaligned\] whose inverse sends a $G$-sheaf $\Q$ on $\wt X$ to  the sheaf of sections of $\wt F:=\wt Q/G\to \wt X/G=X$.

If $M$ is a $G$-module, then $\pi_{\wt X}^{-1}M$ is a (constant) $G$-sheaf on $\wt X$ and, for any $G$-sheaf $\Q$ on $\wt X$, $\La(\wt X,\Q)$ is a $G$-module. The functor $M\mapsto \pi_{\wt X}^{-1}M$ is right adjoint of the functor $\Q\mapsto \La(\wt X,\Q)$.

Now, let us prove the proposition. Let us denote $M_F=\La(\wt X,\phi^{-1}F)$. For any locally constant sheaf $\M$ on $X$ (with corresponding $G$-module $M$) one has
\[\Hom_{\LConst(X)} ( \lc(F),\M)=\Hom_{\Shv(X)}( F,\M)=\Hom_{\GShv(\wt X)}( \phi^{-1}F ,\phi^{-1}\M).\]
Now, $\phi^{-1}\M$ is a constant $G$-sheaf on $\wt X$, hence $\phi^{-1}\M=\pi_{\wt X}^{-1}M$. 
Then
\[ \Hom_{\GShv(\wt X)}(\phi^{-1}F,\phi^{-1}\M)= \Hom_{\text{$G$-mod}}(M_F,M)  \] and we conclude that $M_F$ is the $G$-module corresponding to $\lc(F)$.

If $U$ is a simply connected open subset, then $\phi^{-1}(U)=U\times G$, and then
\[\La(\wt X,\phi^{-1}\ZZ_{U})=\La(\wt X,\ZZ_{\phi^{-1}(U)})=\La(U\times G,\ZZ)=\ZZ[G].\]
\end{proof}

\begin{coro} For any sheaf $F$ on $X$ one has:
\[ \La(X,F)=\La(X,\lc(F) )= \La(\wt X,\phi^{-1}F)_G.\] That is, the cosections of a sheaf $F$ are the coinvariants of the cosections of the lifting of $F$ to the universal covering.
\end{coro}

\begin{proof} The equality  $\La(X,F)=\La(X,\lc(F) ) $ is obtained by adjunction from the commutative diagram
\[\xymatrix{\Mod(\ZZ)\ar[rr]^{\pi^{-1}}\ar[rd]_{\pi^{-1}} &  & \Shv(X)\\ & \LConst(X)\ar[ru]_i &
}\] The second equality follows from Proposition \ref{lc=cosections} and formula \eqref{cos=coinv}.
\end{proof}

\subsection{} Let us now give the derived version of the relation between homology of sheaves and   homology of groups. The equivalence $\LConst(X)=\Mod(G)$ gives an equivalence $$D(\LConst(X))= D(G).$$  Thus, for any $\M\in  D(\LConst(X))$, we shall denote by $M$ its corresponding object in $D(G)$ and viceversa.

Let $$\ZZ\overset\LL{\underset{\ZZ[G]}\otimes} \underline\quad\colon D(G)\to D(\ZZ)$$ be the left derived functor of   $(\quad)_G\colon \Mod(G)\to\Mod(\ZZ)$. We shall use the standard notation
\[ H_i(G,M):=H_i[\ZZ\overset\LL{\underset{\ZZ[G]}\otimes}  M]\] for any $M\in D(G)$. The functor $\ZZ\overset\LL{\underset{\ZZ[G]}\otimes} \underline\quad$ is left adjoint of the functor
\[ \Tr\colon D(\ZZ)\to D(G)\] induced by the exact functor $\Tr\colon\Mod(\ZZ)\to\Mod(G)$.

In other words, the functor $\La(X,\quad)\colon \LConst(X)\to\Mod (\ZZ)$, $\M\mapsto \La(X,\M)$ can be derived by the left (because $\LConst(X)$ has enough projectives), giving a functor
\[ \LL^\text{lc}(X,\quad)\colon D(\LConst(X))\to D(\ZZ)\] which is left adjoint of the functor $\pi^{-1}\colon D(\ZZ)\to D(\LConst(X))$. We shall denote $$H_i^\text{lc}(X,\M):=H_i[\LL^\text{lc}(X,\M)].$$  Via the equivalence $D(\LConst(X))=D(G)$, one has
\[ \aligned \LL^\text{lc}(X,\M)&= \ZZ\overset\LL{\underset{\ZZ[G]}\otimes}M \\ H_i^\text{lc}(X,\M)&=H_i(G,M).\endaligned \]

Now let us consider the inclusion functor $i\colon D(\LConst(X))\to D(X)$. Notice that though $\LConst(X)\to \Shv(X)$ is fully faithful, $i$ is not (in fact, it is proved in \cite{SanchoTorres} that $i$ is fully faithful if and only if $X$ is aspherical). We cannot derive by the left the functor $\lc\colon \Shv(X)\to\LConst(X)$, since $\Shv(X)$ has not enough projectives. However, we can proceed as we did to construct homology:

\begin{prop} The functor $i\colon D(\LConst(X))\to D(X)$ has a left adjoint
\[ \LLc\colon D(X)\to D(\LConst(X)).\] Thus, for any $F\in D(X),\M\in D(\LConst(X))$ one has
\[ \Hom_{D(X)}(F,i(\M))=\Hom_{D(\LConst(X))}(\LLc(F),\M).\]
\end{prop} 

\begin{proof} By Brown representability for the dual, we have to prove that $i$ commutes with direct products. 
Restricting to a simply connected open subset $U$, and taking into account that $\pi_U^{-1}\colon D(\ZZ)\to D(\LConst(U))$ is an equivalence, we conclude by Theorem \ref{derivedproduct}.
\end{proof}

We shall denote: $ \LLc_i(F)=H_i[\LLc(F)].$  By adjunction, $\LLc$ is exact and commutes with direct sums.
For any $F\in D(X)$, $\M\in D(\LConst(X))$ one has an isomorphism
\[ \RR\Hom^\pun (F,i(\M))=\RR\Hom^\pun(\LLc(F),\M).\]
Then, if $F$ is a sheaf, $ \LLc_i(F)=0$ for $i<0$ and $ \LLc_0(F)=\lc(F)$.
For any $F\in D(X)$ one has a morphism $F\to i(\LLc(F))$ and for any $\M\in D(\LConst(X))$ a morphism $\LLc(i(\M))\to\M$ (which is not an isomorphism in general, in contrast with the underived case).  The diagram 
\[\xymatrix{ D(\LConst(X))\ar[rr]^i \ar[rd]_{\LL^\text{lc}(X,\quad)}  & & D(X)\ar[ld]^{\LL(X,\quad)}  \\  & D(\ZZ) &  }\]
is not commutative, but there is a natural morphism $\LL(X,\quad)\circ i\to \LL^\text{lc}(X,\quad)$. Hence,  for any $\M\in D(\LConst(X))$, one has a natural morphism
\[ H_i(X,\M)\to H_i^\text{lc}(X,\M)=H_i(G,M)\] which is not an isomorphism in general (one can prove that it is an isomorphism if $X$ is aspherical, see \cite{SanchoTorres} for the cohomological analog).

On the other hand, one has:
\begin{thm}\label{hsheaves-hgroups} For any $F\in D(X)$:
\[   \LL(X,F)= \LL^{\text{\rm lc}}(X,  \LLc(F)),\qquad \text{and then } H_i(X,F)=H_i^\text{\rm lc}(X,\LLc(F)).\quad\] Equivalently (viewing $\LLc$ as a functor in $D(G)$)
\[   \LL(X,F)= \ZZ\overset\LL\otimes_{\ZZ[G]} \LLc(F),  \qquad \text{and then }  H_i(X,F)=H_i (G,\LLc(F)).\quad \]
\end{thm}

\begin{proof} The result follows from adjunction from 
 the commutativity of the diagram $$\xymatrix{D(\ZZ)\ar[rr]^{\pi^{-1}} \ar[rd]_{\pi^{-1}} & & D(\LConst(X))\ar[ld]_i \\ & D(X) &}.$$  
\end{proof} 

Let us see how to compute $\LLc(F)$ in terms of homology of sheaves in the universal cover, i.e., let us give the derived analog of Proposition \ref{lc=cosections}. First, let us see that $\LLc$ can be computed with (a slight variation of) the standard resolution.

For any complex of sheaves $F$ on $X$, we denote $\lc(F)$ the complex of locally constant sheaves 
\[ \cdots \to \lc(F^n)\to \lc(F^{n+1})\to \cdots\] and one has a morphism of complexes $F\to \lc(F)$ and then a morphism in the derived category $\LLc(F)\to \lc(F)$. We say that $F$ is {\em $\lc$-acyclic} if this morphism is an isomorphism. If $F$ is a sheaf, this is equivalent to say that $\LLc_i(F)=0$ for $i>0$. For any  simply connected open subset $U$, the sheaf $\ZZ_U$ is $\lc$-acyclic, because
\[  \RR\Hom^\pun(\LLc(\ZZ_U),\M)  =\RR\Hom^\pun(\ZZ_U,\M)=\RR\Gamma(U,\M)=\Gamma(U,\M)\] for any locally constant sheaf $\M$. Thus, for any sheaf $F$, let us define
\[ \C_0'F:=\underset{U}\oplus \ZZ_U^{\oplus F(U)}\] where $U$ runs over the set of simply connected open subsets of $X$. One has an epimorphism $\C_0'F\to F$ and $\C_0'F$ is $\lc$-acyclic. Repeating the same construction than for the standard resolution (see \ref{standard-res}), for any complex $F$ one constructs a resolution $\C'_\pun F\to F$, with $\C'_\pun F$ $\lc$-acyclic, and then  
\[ \LLc(F)\simeq \lc(\C'_\pun F).\] By Proposition \ref{lc=cosections},  the complex of $G$-modules corresponding to $\LLc(F)$ by the equivalence $D(\LConst(X))=D(G)$ is $\La(\wt X, \phi^{-1}\C'_\pun F)$. Now, $\phi^{-1}\C_\pun F$ is a resolution of $\phi^{-1}F$ and it is $\La$-acyclic; indeed, from the construction of $\C'_\pun F$, one is reduced to prove that $\phi^{-1}\ZZ_U$ is $\La$-acyclic for any simply connected open subset $U$. But $\phi^{-1}\ZZ_U=\ZZ_{\phi^{-1}(U)}$ and $\phi^{-1}(U)=U\times G$, so $\ZZ_{\phi^{-1}(U)}=\oplus_{g\in G}\ZZ_{U}$, which is a direct sum of $\La$-acyclic sheaves, hence it is $\La$-acyclic. 

%Finally, let us see that the functor $\lc$ takes $\La$-acyclic complexes to $\lc^\text{lc}$-acyclic complexes. Let $F$ be an $\lc$-acyclic complex, i.e. $\LL (X,F)\overset\sim\to \La(X, F)$. Then
%\[ \LL^\text{lc}(X,\lc(F)) = \LL^\text{lc}(X,\LLc(F))=\LL(X,F)=\]
Finally, notice that for any simply connected open subset we have proved $\LLc(\ZZ_U)=\lc(\ZZ_U)\overset{\ref{lc=cosections}}=\ZZ[G]$, which is a projective $G$-module, hence $(\quad)_G$-acyclic.

We have proved then:

\begin{prop} One has 
\[ \LLc_i(F)=H_i(\wt X,\phi^{-1}F)\] as $G$-modules.
%  via  the equivalence $\LConst(X)=\Mod(G)$. More generally, 
%\[ \LLc(X,F)\simeq \La(\wt X,\phi^{-1}\C'_\pun F).\]
Combining with Theorem \ref{hsheaves-hgroups}, one obtains an spectral sequence
\[ H_p (G,H_q(\wt X,\phi^{-1}F))\Rightarrow H_{p+q}(X,F).\]
\end{prop}

%Finally, one has:
%
%\begin{thm} For any $F\in D(X)$ one has:
%\[ \LL(X,F)= \LL^\lc(X,\LLc(F))\] and then an spectral sequence
%\[ H_p^\lc(X,\LLc_q(F))\Rightarrow H_{p+q}(X,F).\] In terms of homology of groups, one has
%\[ \LL(X,F)= \ZZ\overset\LL{\underset{\ZZ[G]}\otimes}\LLc(F)\] and an spectral sequence
%\[ H_p (G,H_q(\wt X,\phi^{-1}F))\Rightarrow H_{p+q}(X,F).\]
%\end{thm}

\end{document}